\documentclass[10pt,a4paper]{article}
\usepackage{amsmath}
\usepackage{environ}
\usepackage{authblk}
\usepackage{etoolbox}
\usepackage{amsthm}
\usepackage{enumerate}
\usepackage[english]{babel}
\usepackage{hyperref}
\usepackage{cite}
\hypersetup{colorlinks,linkcolor={blue},citecolor={blue},urlcolor={blue}}
\usepackage{amsmath,amssymb}
\usepackage{amsfonts}
\usepackage[pdftex]{graphicx}
\usepackage[version=4]{mhchem}
\usepackage{pict2e}
\usepackage[utf8]{inputenc}
\newtheorem{theorem}{Theorem}[section]
\newtheorem{corollary}{Corollary}[section]
\newtheorem{lemma}[theorem]{Lemma}
\theoremstyle{definition}
\newtheorem{definition}{Definition}[section]
\newtheorem{proposition}{Proposition}[section]
\newtheorem{example}{Example}[section]
\theoremstyle{remark}
\newtheorem{remark}{Remark}[section]
\usepackage{mathtools, nccmath}
\usepackage{rotating}
\usepackage{setspace}
\usepackage{etoc}
\usepackage{lipsum}
\usepackage{blindtext}
\usepackage[table,xcdraw]{xcolor}
\usepackage{multirow}
\usepackage{booktabs}
\usepackage{siunitx}
\usepackage{booktabs}
\usepackage{diffcoeff}
\usepackage{outline}
\usepackage{subcaption}
\usepackage{avant} 
\usepackage{amsmath}
\usepackage{multirow}
 \usepackage[table,xcdraw]{xcolor}

\usepackage{array}
\usepackage{rotating}
\usepackage{tikz}
\usepackage{subcaption}
\usepackage{graphicx,wrapfig,lipsum}
\usepackage[a4paper, total={6.5in, 10in}]{geometry}
\newcommand\norm[1]{\left\lVert#1\right\rVert}

\usepackage{lscape}
\usepackage{color, colortbl}
\definecolor{Gray}{gray}{0.9}

\date{}

\begin{document}
\title{ Vieta-Lucas Wavelet based schemes for the numerical solution of the singular models}
\author[1]{Shivani Aeri }
\author[1]{Rakesh Kumar}
\author[2,3]{Dumitru Baleanu}
\author[4*]{Kottakkaran Sooppy Nisar}
\affil[1]{School of Mathematics, Computers and Information Sciences, Central University of Himachal Pradesh, Shahpur Campus, Shahpur 176206, India.  (\textit{e-mail:} shivvani.aeri@gmail.com, rakesh.lect@hpcu.ac.in) }
\affil[2]{Department of Mathematics, Cankaya University, 06810 Ankara, Turkey. (\textit{e-mail:} dumitru@cankaya.edu.tr)}
\affil[3]{Institute of Space Sciences, Magurele-Bucharest, Romania}
\affil[4]{Department of Mathematics, College of Science and Humanities in Alkharj, Prince Sattam Bin Abdulaziz University, 11942, Saudi Arabia. (\textit{e-mail:} n.sooppy@psau.edu.sa)}

\date{}
\maketitle
\begin{abstract}
\noindent In this paper, numerical methods based on Vieta-Lucas wavelets are proposed for solving a class of singular differential equations. The operational matrix of the derivative for Vieta-Lucas wavelets is derived. It is employed to reduce the differential equations into the system of algebraic equations by applying the ideas of the collocation scheme, Tau scheme, and Galerkin scheme respectively. Furthermore, the convergence analysis and error estimates for Vieta-Lucas wavelets are performed. In the numerical section, the comparative analysis is presented among the different versions of the proposed Vieta-Lucas wavelet methods, and the accuracy of the approaches is evaluated by computing the errors and comparing them to the existing findings.
\end{abstract}

\section{Introduction}
\label{sec1}
Singular differential equations (SDEs) have been attracting applied mathematicians for many years because of their applicability in various branches of science, engineering, and technology \cite{stakgold2000boundary,gatica1989singular}. We consider the following SDEs \cite{sabir2021design,sabir2020new}: 
\begin{equation} \label{introeq1}
\rm{ Y''(\zeta) + \frac{\mu}{\zeta}  Y'(\zeta) + f(\zeta, Y(\zeta)) = g(\zeta), ~~ \zeta \in (0,L),  }
\end{equation}
 with initial conditions 
  \begin{equation}
    \rm{ Y(\zeta)\mid_{\zeta=0} = \alpha_{0}, \:  Y'(\zeta)\mid_{\zeta=0} = \alpha_{1}},
 \end{equation}
  or boundary conditions 
 \begin{equation}
    \rm{ Y(\zeta)\mid_{\zeta=0} = \beta_{0}, \:  Y(\zeta)\mid_{\zeta=L} = \beta_{1}},
 \end{equation}
 where  $\rm{g(\zeta)}$ is an analytical function, $\rm{f(\zeta, Y(\zeta))}$ is a continuous real valued function, and $\alpha_{0}$, $\alpha_{1}$, $\beta_{0}$, $\beta_{1}$ are arbitrary constants. Several researchers have been interested in singular models since the available singularity in these models makes computational efforts more tedious. The few well-known singular models which have wider applications in thermodynamics, astrophysics, and atomic physics, are named as  Emden-Fowler model, Lane-Emden model, and Thomas-Fermi model   \cite{sabir2018neuro}. The Runge-Kutta technique, Euler's method, Adams method, Milne predictor-corrector approach, and other traditional methods fail to analyze the singular models efficiently \cite{sabir2021efficient}. SDEs generally appear with variable coefficients, and the solution to these equations can be obtained by the power series method, continued fractions, and integral transforms. The solution to these equations leads to the generation of several special functions \cite{da2019singular}. The existence and uniqueness of the solutions for the SDEs have been investigated by Ford and Pennline \cite{ford2009singular}. The singularities occurring in these equations have encouraged numerous researchers to explore the solutions by proposing novel numerical schemes. To address singular nonlinear problems, Bender et al. \cite{bender1989new} presented a perturbation approach. Russell and Shampine \cite{russell1975numerical} used patch bases,  collocation method, and finite difference method for the treatment of SBVPs. Wazwaz \cite{wazwaz2002new,wazwaz2011variational} respectively employed Adomian decomposition to solve SIVPs and variational iteration method to solve the nonlinear SBVPs. By incorporating a quasi-linearization approach, El-Gebeily and O'Regan\cite{el2007quasilinearization} solved second-order nonlinear SDEs. Yildirim and Ozis \cite{yildirim2007solutions} employed the homotopy perturbation method to solve SIVPs. Pandey \cite{pandey1997finite} used the finite difference method to solve a class of SBVPs. Sabir et al. used artificial neural networking based algorithms to solve various SDEs \cite{guirao2020design, sabir2020intelligence, sabir2021evolutionary}. Researchers have also utilized Chebyshev polynomials \cite{kadalbajoo2005numerical}, Legendre polynomials \cite{el2019numerical}, Laguerre polynomials \cite{zhou2016numerical,baishya2021laguerre}, Jacobi polynomials \cite{el2020novel} and Hermite polynomials \cite{parand2010approximation} to derive the solutions of SDEs. However, less emphasis has been placed on the Vieta-Lucas polynomials (VLPs). Recently,
 VLPs based schemes are used by 
 Agarwal \cite{agarwal2020vieta, agarwal2019some,zhou2017solvability,agarwal2018fractional} and El-Sayed to solve fractional advection-dispersion equation and Heydari et. al \cite{heydari2021vieta} to solve variable order fractional Ginzburg-landau equations. 
\\
In recent decades, wavelets have piqued the interest of many academicians due to their ability in solving differential equations with high precision and minimal processing effort \cite{torrence1998practical,mehra2009wavelets}. Recently, Chebyshev wavelets \cite{yuanlu2010solving,rostami2012comparative}, Legendre wavelets\cite{razzaghi2001legendre}, Hermite wavelets \cite{kumar2021fractional}, Jacobi wavelets \cite{eslahchi2018use}, Gegenbauer wavelets \cite{usman2018efficient} and Lucas wavelets\cite{koundal2020lucas,koundal2022lucas} have been successfully utilized in literature for the simulation of various phenomena of scientific and technical importance. Vieta-Lucas wavelets have not been used in past to solve singular differential equations yet. Recently, some work is reported in the literature on Vieta-Lucas polynomials, but not much attention has been given to Vieta-Lucas wavelets. The wavelets are localized in nature and have compact support so to fill the literature gap, we choose to construct the Vieta-Lucas wavelets.\\
 The main concern of the paper is to propose a novel class of wavelets called as Vieta-Lucas wavelets that can handle various differential equations efficiently. The novelty of the work includes the derivation of generating function and Rodrigues' formula for VLPs. The shifted form of VLPs is used to prepare the Vieta-Lucas wavelets. The operational matrix (OM) of derivatives is proposed to formulate the numerical scheme. The Emden-Fowler and Lane-Emden type SDEs are also solved by the proposed approaches and perform well near the singularity. The accuracy of the methods is analyzed by computing the errors in $L^2$ and $L^{\infty}$ norms.\\
The remaining part is organized as: A brief overview of VLPs is given in section \ref{sec2} by including generating function, Rodrigues' formula, and other significant properties. In order to solve the problems over $[0, L]$, shifted version of VLPs is given in section \ref{sec3}. In section \ref{sec4}, Vieta-Lucas wavelets and their functional approximations are presented, and their operational matrix of the derivative is derived in section \ref{sec5}. In section \ref{sec6} Vieta-Lucas wavelets-based numerical schemes are provided. Section \ref{sec7} deals with the convergence and error estimation for Vieta-Lucas wavelets. In section \ref{sec8}, numerous illustrations are solved by the proposed approaches, and results are provided in section \ref{sec9} that validate the efficiency and reliability of the proposed methods. Finally, the conclusion of the proposed work is demonstrated in section \ref{sec10}.

\section{Vieta-Lucas polynomials and properties} \label{sec2}
In this section, we give brief introduction about Vieta-Lucas polynomials, recurrence relation,  orthogonality property, generating function, Vieta-Lucas differential equation and Rodrigues' formula.   
\begin{definition}
The Vieta-Lucas polynomials $\rm{VL_{m}(\zeta)}$ of degree $\rm{m}$ ($\rm{m \in \mathbb{N} \cup \{0\}}$) can be defined as \cite{horadam2002vieta,robbins1991vieta}:
\begin{equation}
 \rm{VL_{m}(\zeta)= 2 \cos (m\delta)},  
\end{equation}
where $\delta = \arccos{(\frac{\zeta}{2})}$ and $ | \zeta |  \in  [-2,2]$, $\delta \in  [0,\pi].$

\end{definition}
\begin{proposition} \label{p2.1}
The recurrence relation for Vieta-Lucas polynomials $\rm{VL_{m}}(\zeta)$ is given by \cite{horadam2002vieta}:
\begin{equation}
\rm{VL_{m}(\zeta) = \zeta VL_{m-1}(\zeta) - VL_{m-2}(\zeta)},~~ m \geq 2, 
\end{equation}
with $\rm{VL_{0}(\zeta) = 2}$  and 
$\rm{VL_{1}(\zeta) = \zeta}$.
\end{proposition}

 \begin{proposition}
 \label{p2.2}
The Vieta-Lucas polynomials can be represented in terms of power series expansion as \cite{horadam2002vieta}:
\begin{equation}
\label{eq2.2 1}
    \rm{VL_{m}(\zeta) = \sum_{i=0}^{\lfloor m/2 \rfloor} (-1)^{i} \frac{m (m-i-1)!\: }{i! (m-2i)!} \zeta^{m-2i}}, ~~ m \geq 1.
\end{equation}
\end{proposition}
The first few Vieta-Lucas polynomials are given as:
\begin{equation*}
\begin{split}
    \rm{VL_{0}(\zeta)} &= 2 , \\
    \rm{VL_{1}(\zeta)} &= \zeta ,\\
    \rm{VL_{2}(\zeta)} &=\zeta^2 - 2 ,\\
    \rm{VL_{3}(\zeta)} &= \zeta^3 - 3\zeta,\\
    \rm{VL_{4}(\zeta)} &= \zeta^4 - 4\zeta^2 +2,\\
    \rm{VL_{5}(\zeta)} &=  \zeta^5 - 5 \zeta^3 +5 \zeta,\\
    \rm{VL_{6}(\zeta)} &=  \zeta^6 - 6 \zeta^4 +9 \zeta^2 - 2.
\end{split}
\end{equation*}

\begin{proposition}
The Vieta-Lucas polynomials $\rm{VL_{n}(\zeta)}$ and $\rm{VL_{m}(\zeta)}$  defined over $[-2,2]$ are orthogonal in weighted sense with the weight function  $\rm{w(\zeta)= \frac{1}{\sqrt{4 - \zeta^2}}}$ and satisfy the following condition \cite{agarwal2020vieta}:
\begin{equation}
    \rm{\langle VL_{n}(\zeta) ,  VL_{m}(\zeta) \rangle_{w(\zeta)} = \int_{-2}^{2}  VL_{n}(\zeta) VL_{m}(\zeta) w(\zeta)  \,d\zeta=
    \begin{cases}
      4 \pi,   & n = m = 0,  \\ 
      2 \pi,   & n = m \neq 0,  \\ 
       0 ,     & n \neq m.
\end{cases}}
\end{equation}
\end{proposition}
\begin{proposition}
The generating function for Vieta-Lucas polynomials is defined as:
\begin{equation}
    \rm{\sum_{m=0}^{\infty} VL_{m}(\zeta) t^m = \frac{2 - \zeta t}{1- \zeta t + t^2}}. 
\end{equation}
\end{proposition}
\begin{proof}
Since, 
\begin{equation}
    \nonumber
   \rm{ \sum_{m=0}^{\infty} \rm{VL}_{m}(\zeta) t^m = \rm{VL}_{0}(\zeta) t^0 + \rm{VL}_{1}(\zeta) t^1 + \sum_{m=2}^{\infty} \rm{VL}_{m}(\zeta) t^m}.
\end{equation}
Therefore,
\begin{equation*}
\begin{split}
    \sum_{m=0}^{\infty} \rm{VL_{m}(\zeta) t^{m}} &= \rm{2 +\zeta t} + \sum_{m=2}^{\infty} \rm{[\zeta \rm{VL}_{m-1}(\zeta) - \rm{VL}_{m-2}(\zeta)] t^m},\\
   \implies \sum_{m=0}^{\infty} \rm{VL_{m}(\zeta) t^m} &= \rm{ 2 + \zeta t +\zeta} \sum_{m=1}^{\infty} \rm{VL_{m}(\zeta) t^{m+1}} - \sum_{m=0}^{\infty} \rm{VL_{m}(\zeta) t^{m+2}},\\
  \implies \sum_{m=0}^{\infty} \rm{VL_{m}(\zeta) t^m} &=  \rm{ 2 + \zeta t +\zeta t} [\sum_{m=0}^{\infty} \rm{VL_{m}(\zeta) t^{m} - \rm{VL}_{0}(\zeta)] - t^2} \sum_{m=0}^{\infty} \rm{VL_{m}(\zeta) t^{m}}, \\
    \implies \sum_{m=0}^{\infty} \rm{VL_{m}(\zeta) t^m}  &= \rm{ 2 - \zeta t + (\zeta t   - t^2)} \sum_{m=0}^{\infty} \rm{VL_{m}(\zeta) t^{m}},\\
    \implies \sum_{m=0}^{\infty} \rm{VL_{m}(\zeta) t^m} & = \rm{\frac{2 - \zeta t }{1 - \zeta t + t^2}}.
\end{split}
\end{equation*}

\end{proof}

\begin{proposition}\label{p2.4}
The Vieta-Lucas differential equation is defined as \cite{izadi2021application}:
\begin{equation} \label{2a}
   \rm{ (4 - \zeta^{2})\diff[2]Y\zeta -\zeta \diff Y\zeta+m^2Y = 0}, ~~ m \in \mathbb{N}.
\end{equation}
\end{proposition}

\begin{theorem}
The Rodrigues' formula for Vieta-Lucas polynomials $\rm{VL}_{m}(\zeta)$ can be obtained as: 
\begin{equation}
    \rm{VL_{m}(\zeta) =  (-1)^m\:  2\, \frac{m!}{(2m)!} (4 - {\zeta}^2)^{\frac{1}{2}}\frac{d^m}{d\zeta^m}\Big\{(4 - {\zeta}^2)^{m-\frac{1}{2}}\Big\}}.
\end{equation}
\end{theorem}
\begin{proof}
Consider
\begin{equation}
    \nonumber
   \rm{ f(\zeta) = (4 - {\zeta}^2)^{m-\frac{1}{2}}}.
\end{equation}
On differentiating, we obtain
\begin{equation}
\nonumber
   \rm{ f'(\zeta)= - (2m-1) \zeta (4 - {\zeta}^2)^{m-\frac{3}{2}}},
\end{equation}
\begin{equation}
\nonumber
\Rightarrow   \rm{ (4 - {\zeta}^2) f'(\zeta) + (2m-1) \zeta f(\zeta) = 0}.
\end{equation}
On differentiating it (m+1) times, we have
\begin{equation}
  \label{eq2.1}
 \rm{   (4 - {\zeta}^2)\: D^{m+2}f(\zeta) - 3\, \zeta \:D^{m+1}f(\zeta) +(m^2 - 1)\: D^m f(\zeta) = 0, }   ~~~~~ where~  D \equiv \frac{d}{d\zeta}.
\end{equation}
 Let $\rm{ g(\zeta) = (4 - {\zeta}^2)^{\frac{1}{2}}\: D^{m}f(\zeta)\:}$, then \eqref{eq2.1} reduces to
\begin{equation} \nonumber
   \rm{ (4 - \zeta^{2})\, g''(\zeta) - \zeta\,g'(\zeta)+m^2 g(\zeta) = 0},
\end{equation}
which is equivalent to Vieta-Lucas differential equation. Thus, we can choose \:$\rm{VL}_{m}(\zeta) = C_{m}\,g(\zeta)$,\: where $\rm C_{m}$ is a constant to be determined. To find $\rm C_{m}$, the coefficients of $\rm{\zeta^m}$ in $\rm{VL}_{m}(\zeta)$ and $\rm{g(\zeta)}$ are required to be compared.
The coefficient of $\rm{\zeta^m}$ in $\rm{VL}_{m}(\zeta)$ is 1.\\
Since,
\begin{equation*}
\begin{split}
\rm{ g(\zeta) }&= \rm{(4 - {\zeta}^2)^{\frac{1}{2}}\: D^{m}\,(4 - {\zeta}^2)^{m-\frac{1}{2}}}, \\
 & = \rm{(4 - {\zeta}^2)^{\frac{1}{2}}\: \sum_{j=0}^{m} \binom{m}{j}\: D^{m-j}\,(2+\zeta)^{m-\frac{1}{2}}\:D^j\,(2- \zeta)^{m-\frac{1}{2}}},~~\text{ (by Leibniz's rule)}\\
 & =\rm{(-1)^m\: m!\: \sum_{j=0}^{m} \binom{m-\frac{1}{2}}{j} \binom{m-\frac{1}{2}}{m-j} (\zeta-2)^{m-j}\:(\zeta + 2)^j}.
\end{split}
\end{equation*}
Therefore, the coefficient of $\rm{\zeta^m}$ in $\rm{g(\zeta)}$ is
\begin{equation*}
      \rm{(-1)^m\: m!\: \sum_{j=0}^{m} \binom{m-\frac{1}{2}}{j} \binom{m-\frac{1}{2}}{m-j}}.
\end{equation*}
The Vandermonde's convolution modifies the above coefficient of $\rm{\zeta^m}$ to
\begin{equation*}
\label{2.3}
  \rm{ \frac{(-1)^m}{2}\frac{(2m)!}{m!}}.    
\end{equation*}
Thus, $\rm{C_{m} = (-1)^m\:2\; \frac{m!}{(2m)!}}$.
Hence the Rodrigues' formula.
\end{proof}

\begin{remark} \label{remark1}
The zeroes and extreme points of $\rm{VL}_{m}(\zeta)$ in $[-2,2]$ are respectively presented as
$\zeta_j = 2 \cos{\left((j-\frac{1}{2})\frac{\pi}{m}\right)}$ , $\zeta_j=2\cos{\left(j\frac{\pi}{m}\right)}$,\;where j=1,2,3,\dots, m \cite{horadam2002vieta}.
\end{remark}

\begin{proposition}
The function $\zeta^{m}$ can be expressed in terms of $\rm{VL}_{m}(\zeta)$  as:
\begin{equation}
   \rm{\zeta^m} =  \sum_{j=0}^{\lfloor m/2 \rfloor^{\star}} {m \choose j} \rm{VL}_{m-2j}(\zeta),
    \end{equation}
    where '${\star}$' denotes that the last term in the sum is to be halved when m is even.
    \end{proposition}
 \begin{proof}
 Since, 
 \begin{equation}
    \zeta^{m} = (2 \cos{\delta})^{m} = (e^{i \delta}+e^{-i \delta})^{m}.
 \end{equation}
Therefore, the use of binomial expansion gives
 \begin{equation*}
     \begin{split}
     \zeta^{m} &=  e^{i m \delta } + {m \choose 1} e^{i(m-2)\delta} +\dots+{m \choose m-1} e^{-i(m-2)\delta}+ e^{-i m \delta },\\
 &= ( e^{i m \delta }+ e^{-i m \delta }) + {m \choose 1}( e^{i (m-2) \delta }+ e^{-i (m-2) \delta }) +\dots,\\   
 &=  2 \cos{m \delta} + {m \choose 1} 2 \cos{(m-2) \delta} + {m \choose 2} 2 \cos{(m-4) \delta} + \dots,\\
     &= \rm{VL}_{m}(\zeta)  + {m \choose 1} \rm{VL}_{m-2}(\zeta) + {m \choose 2} \rm{VL}_{m-4}(\zeta) + \dots.
     \end{split}
 \end{equation*}

 Hence the result.
\end{proof}

\section{Shifted Vieta-Lucas polynomials} 
\label{sec3}

\begin{definition}
The shifted Vieta-Lucas polynomials defined over $[0,2]$ are denoted by $\rm{VL^{*}_{m}(\zeta)}$ of degree $\rm{m \in \mathbb{N} \cup \{0\}}$ as \cite{heydari2021vieta}:
\begin{equation}
 \rm{VL^{*}_{m}(\zeta) = VL_{m}(2 \zeta - 2 )}.
\end{equation}
\end{definition}
The recurrence relation for shifted Vieta-Lucas polynomials $\rm{VL^{*}_{m}(\zeta)}$ is \cite{heydari2021vieta}:
\begin{equation}
\rm{VL^{*}_{m}(\zeta) = (2 \zeta - 2) VL^{*}_{m-1}(\zeta) - VL^{*}_{m-2}(\zeta)}, 
\end{equation}
provided $\rm{VL^{*}_{0}(\zeta) = 2}$  and 
$\rm{VL^{*}_{1}(\zeta) = 2 \zeta - 2}.$ \\
The shifted Vieta-Lucas polynomials satisfy the following orthogonality property \cite{heydari2021vieta}:
\begin{equation}
    \rm{\langle VL^{*}_{n}(\zeta) ,  VL^{*}_{m}(\zeta) \rangle_{w^{*}(\zeta)} = \int_{0}^{2}  VL^{*}_{n}(\zeta) VL^{*}_{m}(\zeta) w^{*}(\zeta)   \,d\zeta =
    \begin{cases}
      4 \pi,   & n = m = 0,  \\ 
      2 \pi,   & n = m \neq 0,  \\ 
       0 ,     & n \neq m,
\end{cases}}
\end{equation}
where $\rm{w^{*}(\zeta)= \frac{1}{\sqrt{2 \zeta - \zeta^2}}}$ is the weight function of shifted Vieta-Lucas polynomials.

\section{Vieta-Lucas wavelets and function approximation}
\label{sec4}
 Wavelet is a type of function that is derived through the dilation and translation of the mother wavelet. The continuous wavelet family with dilation parameter $h$ and translation parameter $r$ is defined as \cite{ray2018wavelet}:
\begin{equation}
    \gamma_{h,r}(\zeta) = |h|^{-1/2} \gamma(\frac{\zeta-r}{h}), \hspace{0.5 cm} h,r \in \mathbb{R}, h \neq 0. 
\end{equation}
If $h =h_{0}^{-k}, r = s r_{0}h_{0}^{-k}, h_{0} > 1,r_{0} > 0$, then the discrete wavelet family consists of the following members:
\begin{equation}
    \gamma_{k,s}(\zeta) = |h_{0}|^{k/2} \gamma(h_{0} \zeta - s r_{0}), \hspace{0.5 cm} k,s \in \mathbb{Z}^{+},
\end{equation}
where $\gamma_{k,s}(x)$ constitutes the wavelet basis for $L_{2}(\mathbb{R})$. For a specific choice of $h_{0} = 2$ and $r_{0} = 1$, $\gamma_{k,s}(x)$ constitutes an orthonormal basis. 
\begin{definition}
The Vieta-Lucas wavelets $ \rm{\Upsilon_{s,m}(\zeta) =\Upsilon(k,s,m,\zeta) }$ is defined on the interval [0,2) as: 
\begin{equation}\label{4a}
 \rm{\Upsilon_{s,m}}(\zeta) =
    \begin{cases}
      \rm{ 2^{\frac{k}{2}}\:\widehat{\rm{VL}}_{m}(\:2^k \zeta - \Hat{s}\:), }  & \rm{\frac{\Hat{s}-2}{2^k} \leq \zeta < \frac{\Hat{s}+2}{2^k}} ,  \\ 
       0 ,  & \text{Otherwise},
\end{cases}
\end{equation}
where 
\begin{equation}
    \widehat{\rm{VL}}_{m}(\zeta) = 
\begin{cases}
      \frac{1}{\sqrt{\pi}},   &  m=0, \\ 
        \frac{1}{\sqrt{2 \pi}} \rm{VL_{m}}(\zeta),  & m\geq 1,
\end{cases}
\end{equation}
 $m = 0,1,2,\dots,M-1;$ M be the maximum order of Vieta-Lucas polynomials, $s = 1,2,\dots,2^{k-1}$; $k = 1,2,3,\dots;$ $\hat{s}=2(2s-1)$.
\end{definition}

\begin{remark}
Vieta-Lucas wavelets form an orthogonal set with respect to the weight functions $\rm{w_{s}(\zeta) = w(2^{k} \zeta - \Hat{s}) }$. 
\end{remark}
 
 \begin{definition}
 A function $\rm{Y(\zeta)}$ defined over ${L^{2}}_{w_{s}}[0,2]$ can be written in terms of Vieta-Lucas wavelets series as:  
\begin{equation}\label{5a}
   \rm{ Y(\zeta)= \sum_{s=1}^{\infty} \sum_{m=0}^{\infty} \Lambda_{s,m} \Upsilon_{s,m}(\zeta)},
\end{equation}
with
\begin{equation}\label{s1}
   \rm{  \Lambda_{s,m} = \langle Y(\zeta),\Upsilon_{s,m}(\zeta)\rangle_{ w_{s}(\zeta) } = \int_{0}^{2} Y(\zeta) \Upsilon_{s,m}(\zeta) w_{s}(\zeta) \,d\zeta},
\end{equation}
where $\langle * ,* \rangle$ denotes the inner product.
 \end{definition}

The truncated form of Vieta-Lucas wavelet series can be written as:
\begin{equation}
  \rm{  \bar{Y}(\zeta)} \cong \rm{\sum_{s=1}^{2^{k-1}} \sum_{m=0}^{M-1} \Lambda_{s,m} \Upsilon_{s,m}(\zeta) = \rm{\Lambda^{T} \Upsilon(\zeta)}},
\end{equation}
where $\Lambda$ and $\Upsilon(\zeta)$ are $ \eta \times 1$ matrices  in the following form for $\eta = 2^{k-1}M$, and
\begin{align}
  \rm{   \Lambda} &=\rm{ [ \Lambda_{1,0}, \Lambda_{1,1},\dots, \Lambda_{1,M-1}, \Lambda_{2,0},\dots, \Lambda_{2,M-1},\dots, \Lambda_{2^{k-1},0},\dots, \Lambda_{2^{k-1},M-1}]^{T}},\\
   \label{eq4zi}  \rm{ \Upsilon(\zeta)} &= \rm{[\Upsilon_{1,0}(\zeta),\dots,\Upsilon_{1,M-1}(\zeta),\Upsilon_{2,0}(\zeta),\dots,\Upsilon_{2,M-1}(\zeta),\dots,\Upsilon_{2^{k-1},0}(\zeta),\dots,\Upsilon_{2^{k-1},M-1}(\zeta)]^{T}}.
\end{align}
\section{Vieta-Lucas wavelet based operational matrix} 
\label{sec5}
\begin{theorem}
Let $\Upsilon(\zeta)$ be the Vieta-Lucas wavelets vector defined in \eqref{eq4zi}. Then the derivative of the vector $\Upsilon(\zeta)$ can be expressed as:
\begin{equation}
    \frac{\rm{d\Upsilon(\zeta)}}{\rm{d\zeta}} = \rm{D \Upsilon(\zeta)},
\end{equation}
where $D$ is the $\eta \times \eta $ matrix given by \\
\begin{center} 
    $\rm{D} = \begin{pmatrix} 
   \rm{ F} & 0 & \dots & 0 \\
    0 & \rm{F} & \dots & 0 \\
    \vdots &\vdots & \ddots &\vdots \\
    0 & 0 & 0     & \rm{F} 
    \end{pmatrix}$\\
\end{center}

 in which $F$ is $M \times M$ square matrix whose $(u,v)^{th}$ element is defined as:
 \begin{equation}
  \rm{ F_{u,v}} =  \begin{cases}
    \rm{  \frac{2^{k} (u - 1)}{\sqrt{\alpha_{u-1}}\sqrt{\alpha_{v-1}}}} ,   & \rm{u = 2,\dots,M;}\hspace{0.1 cm} \rm{ v = 1,2,\dots,u-1;}\hspace{0.1 cm} \rm{(u+v)}=\text{odd} ,  \\ 
       0 ,  & \text{Otherwise}.
    \end{cases}
\end{equation}
\end{theorem}
\begin{proof}
Using shifted Vieta-Lucas polynomials vector, $\Upsilon(\zeta)$ can be rewritten as
\begin{equation} \label{6b}
   \rm{ \Upsilon_{u}(\zeta)= \Upsilon_{s,m}(\zeta)= 2^{\frac{k}{2}} \sqrt{\frac{1}{2 \pi \alpha_{m}}}VL^{*}_{m}(2^{k-1}\zeta-2s+2) \chi_{[\frac{s-1}{2^{k-2}},\frac{s}{2^{k-2}}]}},
\end{equation}
where $\rm{u = (s-1) M + (m+1)}$; $ \rm{s = 1,2,3,....2^{k-1}}$, $\rm{m = 0,1,2,....M-1}$ and 
\begin{equation}
    \rm{\chi_{[\frac{s-1}{2^{k-2}},\frac{s}{2^{k-2}}]}} = 
\begin{cases}
      1,   &\rm{ \zeta \in [\frac{s-1}{2^{k-2}},\frac{s}{2^{k-2}}]} ,  \\ 
       0 ,  & \text{Otherwise}.
\end{cases}
\end{equation}
On differentiating \eqref{6b} with respect to $\rm{\zeta}$, we obtain
\begin{equation} \label{6c}
   \rm{ \frac{d\Upsilon_{u}(\zeta)}{d\zeta}=  2^{\frac{3k}{2}-1} \sqrt{\frac{1}{2 \pi \alpha_{m}}}\frac{d}{d\zeta}VL^{*}_{m}(2^{k-1}\zeta-2s+2) \chi_{[\frac{s-1}{2^{k-2}},\frac{s}{2^{k-2}}]}},
\end{equation}
since \eqref{6c} vanishes outside the interval $\rm{ \zeta \in [\frac{s-1}{2^{k-2}},\frac{s}{2^{k-2}}]}$. Thus, the nonzero components of Vieta-Lucas wavelets expansion exist only in the interval $\rm{ \zeta \in [\frac{s-1}{2^{k-2}},\frac{s}{2^{k-2}}]}$  i.e, $\rm{\Upsilon_{i}(\zeta)}$ for $i = (s-1) M + 1,(s-1) M + 2,\dots, s M.$ Vieta-Lucas wavelets expansion can now be written as
\begin{equation*}
     \rm{ \frac{d\Upsilon_{u}(\zeta)}{d\zeta} = \sum_{i = (s-1) M + 1}^{s M}  a_{i} \Upsilon_{i}(\zeta)}.
\end{equation*}
Here,
\begin{equation*}
     \rm{ \frac{d}{d\zeta}\Upsilon_{u}(\zeta) = 0, ~~~~ \text{for}\hspace{0.2 cm} u = 1, M+1, 2M+1, \dots,(2^{k-1} - 1)M+1}, ~~ \text{because} ~ \rm{ \frac{d}{d\zeta}VL_{0}(\zeta) = 0} ~ \text{everywhere.}
\end{equation*} 
So, the first row of matrix $F$ is zero.\\
Now, the first derivative of shifted Vieta-Lucas polynomials can be expressed as
\begin{equation} \label{6d}
   \rm{ \frac{d}{d\zeta} VL^{*}_{m}(\zeta) = 2 \sum_{\substack{j = 0\\j+m = odd}}^{M-1} \frac{m}{\alpha_{j}} VL^{*}_{j}(\zeta)}.
\end{equation}
Using \eqref{6d} in \eqref{6c}, we obtain
\begin{equation}
    \rm{ \frac{d\Upsilon_{u}(\zeta)}{d\zeta}} =  \rm{2^{\frac{3k}{2}} \sqrt{\frac{1}{2 \pi \alpha_{m}}} \sum_{\substack{j=0\\j+m = odd}}^{M-1} \rm{\frac{m}{\alpha_{j}} VL^{*}_{m}(2^{k-1}\zeta-2s+2) \chi_{[\frac{s-1}{2^{k-2}},\frac{s}{2^{k-2}}]}},}
\end{equation}
which can be rewritten as
\begin{equation*}
     \rm{ \frac{d\Upsilon_{u}(\zeta)}{d\zeta}  =  2^{k} \sum_{\substack{v = 1\\ u+v = odd}}^{u-1} \frac{u - 1}{\sqrt{\alpha_{u-1}}\sqrt{\alpha_{v-1}}} \Upsilon_{(s-1)M+v}(\zeta).}
\end{equation*}
Therefore
\begin{equation*}
    \rm{  \frac{d\Upsilon_{u}(\zeta)}{d\zeta} = F_{u,v} \Upsilon_{(s-1)M+v}(\zeta), }
\end{equation*}
with
\begin{equation}
  \rm{ F_{u,v}} =  \begin{cases}
    \rm{ 2^{k}  \frac{u - 1}{\sqrt{\alpha_{u-1}}\sqrt{\alpha_{v-1}}}} ,   &\rm{u = 2,\dots,M;\hspace{0.1 cm} v = 1,2,\dots,u-1;\hspace{0.1 cm} (u+v)= \text{odd} },  \\ 
       0 ,  & \text{Otherwise}.
    \end{cases}
\end{equation}
which leads to the desired expression.
\end{proof}
For example, if we select k=2 and M=3, then the discrete members of shifted Vieta-Lucas wavelets can be written as: 
\begin{align*}
\nonumber
  \rm{ \Upsilon_{1}(\zeta) = \Upsilon_{1,0}(\zeta)} = & \begin{cases}
    \rm{\frac{2}{ \sqrt{\pi}}} ,   & 0 \leq \zeta < 1 ,  \\ 
       0 ,  & \text{Otherwise}.
    \end{cases}\\ 
     \rm{ \Upsilon_{2}(\zeta) = \Upsilon_{1,1}(\zeta)} = &  \begin{cases}
    \rm{\frac{ 2 \sqrt{2}}{ \sqrt{\pi}} (2 \zeta - 1)},   & 0 \leq \zeta < 1 ,  \\ 
       0 ,  & \text{Otherwise}.
    \end{cases}\\
      \rm{ \Upsilon_{3}(\zeta) = \Upsilon_{1,2}(\zeta)} = & \begin{cases}
    \rm{\frac{ 2 \sqrt{2}}{ \sqrt{\pi}} (8 \zeta^{2}- 8 \zeta + 1)} ,   & 0 \leq \zeta < 1 ,  \\ 
       0 ,  & \text{Otherwise}.
    \end{cases} \\
     \rm{ \Upsilon_{4}(\zeta) = \Upsilon_{2,0}(\zeta)} = & \begin{cases}
    \rm{\frac{2}{ \sqrt{\pi}} },   & 1 \leq \zeta < 2 ,  \\ 
       0 ,  & \text{Otherwise}.
    \end{cases}
    \\ 
     \rm{ \Upsilon_{5}(\zeta) = \Upsilon_{2,1}(\zeta)} = &  \begin{cases}
    \rm{\frac{ 2 \sqrt{2}}{ \sqrt{\pi}} (2 \zeta - 3),}   & 1 \leq \zeta < 2 ,  \\ 
       0 ,  & \text{Otherwise}.
    \end{cases} \\
      \rm{ \Upsilon_{6}(\zeta) = \Upsilon_{2,2}(\zeta)} = & \begin{cases}
    \rm{\frac{ 2 \sqrt{2}}{ \sqrt{\pi}} (8 \zeta^{2}- 24 \zeta + 17)} ,   & 1 \leq \zeta < 2 ,  \\ 
       0 ,  & \text{Otherwise}.
    \end{cases} 
\end{align*}
As a result, the first order derivatives of the shifted Vieta-Lucas wavelets over the domain $[0,2]$ are:
\begin{align*}
  \rm{  \frac{d\Upsilon_{1}}{d\zeta}} & = 0,\\
   \rm{ \frac{d\Upsilon_{2}}{d\zeta}} & = \frac{ 4 \sqrt{2}}{ \sqrt{\pi}} = 2 \sqrt{2}~ \Upsilon_{1},\\
   \rm{ \frac{d\Upsilon_{3}}{d\zeta}} & = \frac{ 16 \sqrt{2}}{ \sqrt{\pi}} (2 \zeta - 1 ) = 8 ~\Upsilon_{2},\\
    \rm{  \frac{d\Upsilon_{4}}{d\zeta}} & = 0, \\
   \rm{ \frac{d\Upsilon_{5}}{d\zeta}} & = \frac{ 4 \sqrt{2}}{ \sqrt{\pi}} = 2 \sqrt{2}~ \Upsilon_{4},\\
   \rm{ \frac{d\Upsilon_{6}}{d\zeta}} & = \frac{ 16 \sqrt{2}}{ \sqrt{\pi}} (2 \zeta - 3 ) = 8 ~\Upsilon_{5}.
\end{align*}
So, the matrix D is as follows:
\begin{center} 
    $\rm{D} = 
    \begin{pmatrix} 
    0 & 0  & 0 & 0 & 0  & 0 \\
    2 \sqrt{2} & 0  & 0 & 0 & 0  & 0 \\
    0 & 8  & 0 & 0 & 0  & 0\\
   0 & 0  & 0 & 0 & 0  & 0 \\
   0 & 0  & 0 &  2 \sqrt{2} & 0  & 0 \\
    0 & 0  & 0 & 0 & 8  & 0\\
    \end{pmatrix}$.\\
\end{center}

\begin{corollary}
The $m^{th}$ order differential OM can be achieved as:
\begin{equation}\label{1st}
    \rm{\frac{d^{(m)}\Upsilon(\zeta)}{d\zeta^{(m)}}= D^{(m)}\Upsilon(\zeta)},
\end{equation}
and $D^{(m)}$ denotes the  $m^{th}$ order derivative of $D$.
\end{corollary}
\section{Numerical Scheme}
\label{sec6}

Consider the most general appearance of second order singular differential equations as
\begin{equation}\label{7a}
\rm{ Y''(\zeta) + \frac{\mu}{\zeta}  Y'(\zeta) + f(\zeta, Y(\zeta)) = g(\zeta), \hspace{0.5 cm} 0\leq\zeta\leq L}.
\end{equation}
with 
  \begin{equation}\label{7b}
    \rm{ Y(\zeta)|_{\zeta=0} = \alpha_{0}, \:  Y'(\zeta)|_{\zeta=0} = \alpha_{1}},
 \end{equation}
  or 
 \begin{equation}\label{7c}
    \rm{ Y(\zeta)|_{\zeta=0} = \beta_{0}, \:  Y(\zeta)|_{\zeta=L} = \beta_{1}}.
 \end{equation}
 Let $\rm{\bar{Y}(\zeta)}$ be the Vieta-Lucas wavelet series approximation to the solution of equations \eqref{7a}-\eqref{7c}
 \begin{equation}\label{7d}
    \rm{ 
\bar{Y}(\zeta)=\sum_{s=1}^{2^{k-1}} \sum_{m=0}^{M-1} \Lambda_{s,m} \Upsilon_{s,m}(\zeta)= \Lambda^{T} \Upsilon(\zeta)}.
 \end{equation} 
 Now by using (\ref{1st}), we obtain  
  \begin{equation}\label{7e}
     \rm{\bar{Y}'(\zeta)= \Lambda^{\rm{T}} D \Upsilon(\zeta)} ~~~~ \text{and} ~~~~ \rm{\bar{Y}''(\zeta)= \Lambda^{\rm{T}} D^{(2)} \Upsilon(\zeta)}.
 \end{equation} 
 The residual function $\rm{R(\zeta)}$  can be obtained by using equations \eqref{7d} and \eqref{7e} in equation \eqref{7a} as 
 \begin{equation}\label{7f}
\rm{R(\zeta)= \Lambda^{\rm{T}} D^{(2)} \Upsilon(\zeta) + \frac{\mu}{\zeta} \Lambda^{\rm{T}} D \Upsilon(\zeta)} + f(\zeta, \Lambda^{\rm{T}} \Upsilon(\zeta))-g(\zeta).     
 \end{equation}
The corresponding updation in equations \eqref{7a} and \eqref{7b} are as follows
\begin{equation}\label{7h}
   \rm{ \Lambda^{T} \Upsilon(0)=\alpha_{0}, ~~~  
   \Lambda^{T} D \Upsilon(0)= \alpha_{1}},
 \end{equation}  
  or 
 \begin{equation}\label{7i}
   \rm{ \Lambda^{T} \Upsilon(0) = \beta_{0}, ~~~ 
   \Lambda^{T} \Upsilon(l)=\beta_{1}}.
 \end{equation}
  When $\rm{R(\zeta)}$ is zero, the exact solution is obtained, although it is quite difficult to get $\rm{R(\zeta)}$ identically zero. So our primary focus is to make the  residual value  as small as possible. Thus, we utilize the weighted residual methods in order to minimize the residual function. Now, the weighted residual equation is written as
 \begin{equation}
     \rm{\langle W , R(\zeta) \rangle} = \int_{0}^{2} W R(\zeta)  \,d\zeta  = 0, ~~  
 \end{equation}

 where $\rm{W}$ is the weighted function in integral sense.\\ 
 Thus, three weighted residual approaches are presented which are based on the different choices of $\rm{W}$.  \\ \\
 \textbf{\underline{Approach - I }: Collocation Approach } \\
 In this approach, the weighted function is chosen as Dirac delta($\delta$) function which vanishes everywhere except at the collocation points, this yields
 \begin{equation}
 \nonumber
      \rm{  \langle  \delta(\zeta - \zeta_{i}) , R(\zeta) \rangle = \int_{0}^{2} \delta(\zeta - \zeta_{i}) R(\zeta)  \,d\zeta, ~~  i= 1,2,3,\dots,2^{k-1}M-2, }
 \end{equation}
 which gives
 \begin{equation}
 \label{7c0}
  \rm{ R(\zeta_{i})} =0, \hspace{0.5 cm} i= 1,2,3,\dots,2^{k-1}M-2.
 \end{equation}
 Here, the extrema of Vieta-Lucas polynomials are chosen as the collocation points.
Now, the $\rm{2^{k-1}M}$ system of equations (equation \eqref{7c0} with conditions \eqref{7h} or \eqref{7i}) yields the values of unknown coefficients, and thus the required solution. The solution achieved in this sense will be called as $\rm{Y_{VLWC}(\zeta)}$ solution.\\ \\
\textbf{\underline{Approach - II }: Tau Approach}\\ 
In this approach, the weighted function is chosen to be the same as the Vieta-Lucas wavelets $\Upsilon_{i}$, which yields $\rm{(2^{k-1}M-2)}$ nonlinear equations as
\begin{equation}\label{7g}
  \rm{  \langle  \Upsilon_{i} , R(\zeta) \rangle = \int_{0}^{2} \Upsilon_{i} ~ R(\zeta) \,d\zeta =0 , \hspace{0.7 cm} i= 1,2,3,\dots,2^{k-1}M-2.} 
\end{equation}
Thus, solving $\rm{2^{k-1}M}$ equations (equation \eqref{7g} with \eqref{7h} or \eqref{7i}), we get the unknown coefficients which leads to the appropriate solution, called as $\rm{Y_{VLWT}(\zeta)}$. \\ \\
\textbf{\underline{Approach - III }: Galerkin  Approach} \\ 
The main idea behind this approach is to expand the solutions not only in terms of usual Vieta-Lucas wavelets expansion, but with some combinations of Vieta-Lucas wavelets which satisfy the boundary requirements. In this approach, the weighted functions are taken as the trial series solution, which gives 
 \begin{equation}\label{7g3}
   \rm{  \langle \Psi_{i} , R(\zeta) \rangle = \int_{0}^{2} \Psi_{i} ~ {R}(\zeta) \,d\zeta =0 , \hspace{0.7 cm} i= 1,2,3,\dots,2^{k-1}M,}
  \end{equation} 
 where the trial series solution for Galerkin approach is written as 
\begin{equation}\label{7g1}
     \rm{{Y_{VLWG}}(\zeta) =  \Lambda^{T} \Psi(\zeta) },
 \end{equation}
 and $\rm{\Psi(\zeta) = \nu(\zeta) \Upsilon(\zeta)}$, where $\rm{\nu(\zeta)}$ be the trial function. 
Now, solving $\rm{2^{k-1}M}$ system of equations from \eqref{7g3} provide the values of unknown coefficients, and thus we get the required solution $\rm{{Y_{VLWG}}(\zeta)}$.

\section{Convergence and error bound estimation}
\label{sec7}
\begin{theorem} 
Let $\rm{Y(\zeta)} \in L^{2}_{w_{s}}[0,2]$ then the Vieta-Lucas series defined in \eqref{5a} converges to $\rm{Y(\zeta)}$ by the use of Vieta-Lucas wavelets i.e.
\begin{equation*}
        \rm{ Y(\zeta) = \sum_{s=1}^{\infty} \sum_{m=0}^{\infty} \Lambda_{s,m} \Upsilon_{s,m}(\zeta)}.
    \end{equation*}
\end{theorem}
\begin{proof}
Suppose $\rm{Y(\zeta)} \in L^{2}_{w_{s}}[0,2]$, where $L^{2}_{w_{s}}[0,2]$ be the Hilbert space and $\rm{\Upsilon_{s,m}}$ defined in \eqref{4a} forms an orthonormal basis with respect to weight function $w_{s}(\zeta) = w(2^{k} \zeta - \Hat{s}). $\\
Consider,  $\rm{Y(\zeta) = \sum_{m=0}^{M-1} \Lambda_{s,m} \Upsilon_{s,m}(\zeta) }$ where $\rm{\Lambda_{s,m} = \langle Y(\zeta),\Upsilon_{s,m}(\zeta)\rangle_{w_{s}(\zeta)}}$ for a fixed $\rm{s}$ and $\rm{\Upsilon_{s,m}(\zeta) = \Upsilon_{j}(\zeta)}$, $\rm{\varsigma_{j} = \langle Y(\zeta),\Upsilon_{j}(\zeta)\rangle}$.
 The sequence of partial sum $\{\rm{\rho_{s}}\}$ of $\{\rm{\varsigma_{j} \Upsilon_{j}(\zeta)}\}_{j \geq 1}$ is defined as $\rm{\rho_{s} = \sum_{j=1}^{s} \varsigma_{j} \Upsilon_{j}(\zeta)}.$\\
Now,
\begin{equation*}
  \rm{  \langle Y(\zeta),\rho_{s} \rangle =  \langle Y(\zeta), \sum_{j = 1}^{s} \varsigma_{j} \Upsilon_{j}(\zeta) \rangle = \sum_{j = 1}^{s} |\varsigma_{j}|^{2}}. 
\end{equation*}
 For s $>$ m, we assert that
\begin{equation*}
   \rm{ \norm{\rho_{s} - \rho_{m} } ^{2} = \sum_{j = m+1}^{s} |\varsigma_{j}|^{2}}. 
\end{equation*}
Now, 
\begin{equation*}
  \rm{  \norm{ \sum_{j = m+1}^{s} \varsigma_{j} \Upsilon_{j}(\zeta) }^{2} =  \langle \sum_{j = m+1}^{s} \varsigma_{j} \Upsilon_{j}(\zeta),\sum_{j = m+1}^{s} \varsigma_{j} \Upsilon_{j}(\zeta) \rangle = \sum_{j = m+1}^{s} |\varsigma_{j}|^{2},\hspace{0.2 cm} s > m.  }
\end{equation*}
Therefore,
\begin{equation*}
   \rm{  \norm{ \sum_{j = m+1}^{s} \varsigma_{j} \Upsilon_{j}(\zeta) }^{2} = \sum_{j = m+1}^{s} |\varsigma_{j}|^{2},\hspace{0.2 cm} s > m.}
\end{equation*}
From Bessel's inequality, we know $\rm{\sum_{j = 1}^{\infty} |\varsigma_{j}|^{2}}$  is convergent.\\
Thus we have 
\begin{equation*}
     \rm{  \norm{ \sum_{j = m+1}^{s} \varsigma_{j} \Upsilon_{j}(\zeta) }^{2}} \rightarrow 0 \hspace{0.5 cm} \text{as}  \hspace{0.2 cm}s \rightarrow \infty. 
\end{equation*}
Which implies 
\begin{equation*}
     \rm{  \norm{ \sum_{j = m+1}^{s} \varsigma_{j} \Upsilon_{j}(\zeta) }} \rightarrow 0 ,
\end{equation*}
and $\{ \rm{\rho_{s}}\}$ is a Cauchy sequence that converges to $\rm{\rho}$(say).\\
Thus,
\begin{equation*}
\begin{split}
    \rm{  \langle \rho - Y(\zeta) , \Upsilon(\zeta) \rangle}  &=  \rm{\langle \rho , \Upsilon(\zeta) \rangle - \langle Y(\zeta) , \Upsilon(\zeta) \rangle ,}\\
      &=  \rm{\langle \lim_{s \rightarrow \infty} \rho_{s} , \Upsilon(\zeta) \rangle - \varsigma_{j} ,}\\
      &=  \rm{\varsigma_{j} - \varsigma_{j}}.\\
\end{split}
\end{equation*}
Which implies 
\begin{equation*}
     \rm{  \langle \rho - Y(\zeta) , \Upsilon(\zeta) \rangle   =  0}.
\end{equation*}
Therefore, $\rm{Y(\zeta) = \rho}$ and $\rm{\sum_{j = 1}^{s} \varsigma_{j} \Upsilon_{j}(\zeta)}$ converges to $\rm{Y(\zeta)}$ for $s \rightarrow \infty$.
\end{proof}
\begin{theorem} \label{8b}
Let $Y(\zeta)$ be a second order square integrable function defined over $[0,2]$ with bounded second order derivative say $|Y''(\zeta)|\leq H$ for some constant H. Then 
$Y(\zeta)$ can be expanded as an infinite sum of Vieta Lucas wavelets and the series converges to $Y(\zeta)$ uniformly, that is 
    \begin{equation*}
       \rm{  Y(\zeta) = \sum_{s=1}^{\infty} \sum_{m=0}^{\infty} \Lambda_{s,m} \Upsilon_{s,m}(\zeta)},
    \end{equation*}
    where $\rm{ \Lambda_{s,m} = \langle Y(\zeta),\Upsilon_{s,m}(\zeta)\rangle_{{L^{2}}_{w_{s}}[0,2]} }.$
\end{theorem}
\begin{proof}
From \eqref{s1}, we have
\begin{align*}
   \rm{ \Lambda_{s,m} = \langle Y(\zeta),\Upsilon_{s,m}(\zeta)\rangle_{{L^{2}}_{w_{s}}[0,2]}  = \int_{0}^{2} Y(\zeta) \Upsilon_{s,m}(\zeta) w_{s}(\zeta) \,d\zeta}, \\ 
   \rm{ = \int_{\frac{\Hat{s}-2}{2^{k}}}^{\frac{\Hat{s}+2}{2^{k}}} Y(\zeta) 2^{\frac{k}{2}} \sqrt{\frac{1}{2 \pi}} {\rm{VL_{m}}}(2^{k}\zeta-\Hat{s})\, w_{s}(2^{k}\zeta-\Hat{s}) \,d\zeta},
\end{align*}
where $\Hat{s}= 2(2s-1)$. Substituting  $2^{k}\zeta-\Hat{s}= 2 \cos{\delta} $ and from the definition of Vieta-Lucas polynomial, we obtain 
\begin{align*}
  \rm{ \Lambda_{s,m}} =& \rm {2^{\frac{-k}{2}} \sqrt{\frac{1}{2 \pi}} \int_{0}^{\pi} Y\left(\frac{2 \cos{\delta}+\Hat{s}}{2^{k}}\right)  \rm{VL_{m}}}(\delta)\, w_{s}(\delta) \,d\delta, \\
    =& \rm{2^{\frac{-k}{2}} \sqrt{\frac{1}{2 \pi}} \int_{0}^{\pi} Y\left(\frac{2 \cos{\delta}+\Hat{s}}{2^{k}}\right)  2 \cos{m \delta}\, \frac{1}{\sqrt{4-(2\cos{\delta})^2}} \, 2 \sin{\delta} \,d\delta}, \\
 =& \rm{2^{\frac{-k}{2}} \sqrt{\frac{1}{2 \pi}} \int_{0}^{\pi} Y\left(\frac{2 \cos{\delta}+\Hat{s}}{2^{k}}\right)  2 \cos{m \delta}\, \frac{1}{2\sqrt{1-\cos^2{\delta}}} \, 2 \sin{\delta} \,d\delta },\\
  =& \rm{ 2^{\frac{-k+1}{2}} \sqrt{\frac{1}{\pi}} \int_{0}^{\pi} Y\left(\frac{2 \cos{\delta}+\Hat{s}}{2^{k}}\right)   \cos{m \delta}\, d\delta}, 
\end{align*}
Using the integration by parts, we get
 \begin{align*}
     \rm{ \Lambda_{s,m}} =& \rm{\frac{2^{\frac{-3k+1}{2}}}{m} \sqrt{\frac{1}{\pi}} \int_{0}^{\pi} Y'\left(\frac{2 \cos{\delta}+\Hat{s}}{2^{k}}\right)   2 \sin{m \delta} \sin{\delta}\, d\delta}, \\
      =& \rm{ \frac{2^{\frac{-3k+1}{2}}}{m} \sqrt{\frac{1}{\pi}} \int_{0}^{\pi} Y'\left(\frac{2 \cos{\delta}+\Hat{s}}{2^{k}}\right)  (\cos{(m-1)\delta}-\cos{(m+1)\delta})\, d\delta },\\
       = & \rm{ \frac{2^{\frac{-3k+1}{2}}}{m} \sqrt{\frac{1}{\pi}}  \left( \int_{0}^{\pi}Y'\left(\frac{2 \cos{\delta}+\Hat{s}}{2^{k}}\right)  \cos{(m-1)\delta}\, d\delta - \int_{0}^{\pi}Y'\left(\frac{2 \cos{\delta}+\Hat{s}}{2^{k}}\right)   \cos{(m+1)\delta}\, d\delta\right), }
 \end{align*}
 \begin{equation}
     \label{8c1}
      \implies  \rm{ |\Lambda_{s,m}|} \leq  ~\rm{ \frac{2^{\frac{-3k+1}{2}}}{m} \sqrt{\frac{1}{\pi}}  (|I_{1}|+|I_{2}|)},
 \end{equation}

     where $\rm{I_{1}} = \int_{0}^{\pi} Y'\left(\frac{2 \cos{\delta}+\Hat{s}}{2^{k}} \right)  \cos{(m-1)\delta}\, d\delta $ and $\rm{I_{2}}= \int_{0}^{\pi}Y'\left(\frac{2 \cos{\delta}+\Hat{s}}{2^{k}}\right)   \cos{(m+1)\delta}\, d\delta $. Next we estimate $\rm{I_{1}}$ and  $\rm{I_{2}}$ respectively.
    On Integrating $\rm{I_{1}}$ by parts, we have
      \begin{align*}
       \rm{ I_1} =& \rm{\frac{2^{-k+1}}{m-1} \int_{0}^{\pi}Y''\left(\frac{2 \cos{\delta}+\Hat{s}}{2^{k}}\right)  \sin{(m-1)\delta} \sin{\delta}\, d\delta}, \\
       \rm{ |I_1|} \leq & \rm{\frac{2^{-k+1}}{m-1} \int_{0}^{\pi}\left|Y''\left(\frac{2 \cos{\delta}+\Hat{s}}{2^{k}}\right)\right|  \, d\delta },
     \end{align*}
     which gives
     \begin{equation} \label{8c2}
         \rm{ |I_1| }\leq  \rm{ \frac{ H \pi 2^{-k+1}}{m-1}}.
     \end{equation}
     Similarly on Integrating $\rm{I_{2}}$ , we obtain 
     \begin{align*}
       \rm{ I_2} =& \rm{\frac{2^{-k+1}}{m+1} \int_{0}^{\pi}Y''\left(\frac{2 \cos{\delta}+\Hat{s}}{2^{k}}\right)  \sin{(m+1)\delta} \sin{\delta}\, d\delta}, 
     \end{align*}
     which gives
     \begin{equation} \label{8c3}
          \rm{  |I_2|} \leq  \rm{ \frac{ H \pi 2^{-k+1}}{m+1}}.
     \end{equation}
    By using \eqref{8c2} and \eqref{8c3} in \eqref{8c1}, we obtain
    \begin{align*}
        \rm{ | \Lambda_{s,m}|\leq \frac{H \sqrt{\pi} \, 2^{\frac{-5k+5}{2}}}{m^2 -1}}, \hspace{0.5 cm} m>1.
    \end{align*}
    Since $\rm{s\leq 2^{k-1}}$ and $\rm{s \geq 1}$. Therefore we get
    \begin{equation}\label{8c4}
         \rm{| \Lambda_{s,m}|\leq \frac{ H \sqrt{\pi} \,}{s^{\frac{5}{2}}(m^2 -1)}}. 
    \end{equation} 
    For m = 1, we have 
     \begin{align*}
        \rm{ | \Lambda_{s,1}|\leq \frac{\sqrt{\pi} \,}{2~s^{\frac{3}{2}}(m^2 -1)} \max_{0 \leq \zeta \leq 2} |\rm{Y'(\zeta)}|}. 
    \end{align*}  
    Also,
    \begin{align*}
       \left| \rm{ \sum_{s=1}^{\infty} \sum_{m=0}^{\infty}  \Lambda_{s,m} \Upsilon_{s,m}(\zeta)} \right| &\leq \rm{\left| \sum_{s=1}^{\infty}  \Lambda_{s,0} \Upsilon_{s,0}(\zeta) \right| + \sum_{s=1}^{\infty} \sum_{m=1}^{\infty} |\Lambda_{s,m}|
       |\Upsilon_{s,m}(\zeta)|} \\
       & \leq \rm{\left|\sum_{s=1}^{\infty}  \Lambda_{s,0} \Upsilon_{s,0}(\zeta) \right| + \sum_{s=1}^{\infty} \sum_{m=1}^{\infty} |\Lambda_{s,m}|} < \infty.
     \end{align*}
    which implies that the series $\rm{ \sum_{s=1}^{2^{k-1}} \sum_{m=0}^{M-1} \Lambda_{s,m}}$ is absolutely convergent. Thus the series\\ $\rm{ \sum_{s=1}^{2^{k-1}} \sum_{m=0}^{M-1} \Lambda_{s,m} \Upsilon_{s,m}(\zeta)}$ converges to $\rm{Y(\zeta)}$ uniformly.
\end{proof}

\begin{lemma}\label{8c}
Let $f(\zeta)$ be a continuous, positive, decreasing function for $\zeta \geq m$ if  $f(k) = \Lambda_{k}$, provided that $\sum \Lambda_{s}$ is convergent and $R_{s} = \sum_{k = s+1}^{\infty} \Lambda_{k}$, then $R_{s} \leq \int_{s}^{\infty} f(\zeta) d\zeta$    \cite{stewart2012single}. 
\end{lemma}
\begin{theorem}
For the Vieta-Lucas wavelets expansion, if $Y(\zeta)$ satisfies the theorem \eqref{8b}, then the following error estimate holds for $M >2$,
\begin{equation*}
  \rm{  \norm{Y(\zeta) - \bar{Y}(\zeta)}_{w_{s}}} < H \sqrt{\rm{ \pi   \frac{(M^{2} - 2M) ln(M) - M^{2} ln(M-2) + (2 ln(M-2)-2)M+2}{(2^{5(2^{k-1} - 1)} 5 \ln{2} )~ 4M(M-2)} }}.
\end{equation*}
\begin{proof}
Considering the Vieta-Lucas wavelets expansion as
\begin{equation*}
     \rm{  \bar{Y}(\zeta) = \sum_{s=1}^{2^{k-1}} \sum_{m=0}^{M-1} \Lambda_{s,m} \Upsilon_{s,m}(\zeta)}.
\end{equation*}
From \eqref{4a}, by using the orthogonality property of Vieta-Lucas wavelets with respect to weight function $w_{s}(\zeta)$, we obtain
\begin{equation*}
    \rm{ \norm{Y(\zeta) - \bar{Y}(\zeta)}^{2}_{w_{s}} = \sum_{s=2^{k-1}+1}^{\infty} \sum_{m=M}^{\infty} \Lambda^{2}_{s,m}}.
\end{equation*}
By using theorem \eqref{8b}, it can be expressed as
\begin{equation*}
    \rm{ \norm{Y(\zeta) - \bar{Y}(\zeta)}^{2}_{w_{s}} < H^{2} \pi \sum_{s=2^{k-1}+1}^{\infty} \sum_{m=M}^{\infty} \frac{1}{2^{5k-5}(m^2 - 1)^{2}}}.
\end{equation*}
From lemma \eqref{8c}, we obtain
\begin{equation*}
\begin{split}
     \rm{\norm{Y(\zeta) - \bar{Y}(\zeta)}^{2}_{w_{s}}} &< \rm{H^{2} \pi \int_{s=2^{k-1}}^{\infty} \int_{m=M - 1}^{\infty} \frac{1}{2^{5 \zeta - 5}(z^2 - 1)^{2}} d\zeta dz ,} \\
     &= \rm{ H^{2} \pi \left(  \int_{s=2^{k-1}}^{\infty} \frac{1}{2^{5 \zeta - 5}} d\zeta \right) \left(\int_{m=M - 1}^{\infty} \frac{1}{(z^2 - 1)^{2}} dz \right) ,}\\
     &=\rm{ H^{2} \pi \left(\frac{1}{2^{5(2^{k-1} - 1)} 5 \ln{2}} \right) \left( \frac{(M^{2} - 2M) ln(M) - M^{2} ln(M-2) + (2 ln(M-2)-2)M+2}{4M(M-2)} \right)}.
\end{split}
\end{equation*}
which completes the proof.
\end{proof}
\end{theorem}
\section{Numerical Examples}
\label{sec8}

This section contains numerical examples that demonstrate the efficiency and reliability of the presented schemes.

\begin{example}
\label{ex1}
Consider the non-homogeneous SDE \cite{nasab2013wavelet}:
\begin{equation}\label{ex11} 
   \rm{ Y''(\zeta)+\frac{1}{\zeta}Y'(\zeta)  ={\left(\frac{8}{8-\zeta^{2}}\right)}^{2},\hspace{0.5 cm} 0\leq \zeta \leq 1,
}\end{equation}
with 
\begin{equation}
   \rm{Y(\zeta)|_{\zeta = 0} = 0, \hspace{0.4 cm} Y'(\zeta)|_{\zeta = 0} = 0. }\\
\end{equation}
The exact solution is
\begin{equation*}
   \rm{Y_{Exact}(\zeta) = 2 \log{\left(\frac{8}{8-\zeta^2}\right)}}. 
\end{equation*}
The following solutions are obtained by using the proposed numerical schemes  at $\eta$ = 12:
\begin{equation*}
\begin{split}
   \rm{ Y_{VLWC}(\zeta)} &= (1.60 \times 10^{-16}) - (1.18 \times 10^{-16}) \zeta +  \dots + (1.82 \times 10^{-5}) \zeta^{11} ,\\
 \rm{Y_{VLWT}(\zeta)} &=   (1.83 \times 10^{-16}) + (4.51 \times 10^{-17}) \zeta +  \dots + (8.79 \times 10^{-4}) \zeta^{11},\\
 \rm{ Y_{VLWG}(\zeta)} &= 0.25 \zeta^{2} -0.08 \zeta^{3} + 0.40 \zeta^{4} + \dots - (9.75 \times 10^{-4} \zeta^{13}).
 \end{split}
\end{equation*}

\end{example}

\begin{example}
\label{ex2}
Take the nonlinear SDE as \cite{wazwaz2001reliable}:
\begin{equation}
   \rm{ Y''(\zeta)+ \pi^{3} \frac{(Y^{2}(\zeta))}{\sin{(\pi \zeta)}}  = 0,\hspace{0.5 cm} 0< \zeta < 1,
}\end{equation}
with boundary restrictions
\begin{equation}
   \rm{Y(\zeta)|_{\zeta = 0} = 0, \hspace{0.4 cm} Y(\zeta)|_{\zeta = 1} = 0. }
   \end{equation}
   The exact solution is
\begin{equation*}
   \rm{Y_{Exact}(\zeta) = \frac{\sin{(\pi \zeta)}}{\pi}}. 
\end{equation*}
   The following solutions are obtained for $\eta$ = 6:
\begin{equation*}
\begin{split}
   \rm{ Y_{VLWC}(\zeta)} &= (1.38 \times 10^{-16}) + 0.99 \zeta + 0.10 \zeta^{2} -2.20 \zeta^{3} + 1.10 \zeta^{4} - (7.04 \times 10^{-13}) \zeta^{5},\\
 \rm{Y_{VLWT}(\zeta)} &= (-2.77 \times 10^{-17}) + 0.93 \zeta + 0.64 \zeta^{2} -3.62 \zeta^{3} + 2.56 \zeta^{4} -0.51 \zeta^{5},\\
 \rm{ Y_{VLWG}(\zeta)} &= 1.00 \zeta - 0.03 \zeta^{2} - 1.39 \zeta^{3} - 0.75 \zeta^{4} + 1.96 \zeta^{5} -0.90 \zeta^{6} + 0.12 \zeta^{7}.
 \end{split}
\end{equation*}
   \end{example}

\begin{example}\label{ex3}
Consider the non homogeneous Emden-Fowler type SDE \cite{wazwaz2002new}:
\begin{equation}
   \rm{ Y''(\zeta)+\frac{8}{\zeta}Y'(\zeta) + \zeta Y(\zeta) =\zeta^{5} - \zeta^{4} + 44 \zeta^{2} - 30 \zeta,\hspace{0.5 cm}  \zeta \geq 0,
}\end{equation}
with 
\begin{equation}
   \rm{Y(\zeta)|_{\zeta = 0} = 0, \hspace{0.4 cm} Y'(\zeta)|_{\zeta = 0} = 0. }
   \end{equation}

The following solutions are obtained for $\eta$ = 5:
\begin{equation*}
\begin{split}
   \rm{ Y_{VLWC}(\zeta)} &=\rm{(5.55 \times 10^{-16}) +  ( 9.32 \times 10^{-15}) \zeta^2 - \zeta^3 + \zeta^4},\\
    \rm{Y_{VLWT}(\zeta)} &=\rm{ \zeta^4 - \zeta^3},\\
\rm{ Y_{VLWG}(\zeta)} &= \rm{\zeta^4 - \zeta^3}.
 \end{split}
\end{equation*}
 The above solutions are computed in the interval $0\leq \zeta \leq 2$ and it is observed that both $\rm{Y_{VLWT}}$ and $\rm{Y_{VLWG}}$ yield the exact solution. 
  \end{example}


\begin{example} 
\label{ex4}
Consider the nonlinear Emden-Fowler type SDE \cite{wazwaz2005adomian}:
\begin{equation}
   \rm{ Y''(\zeta)+\frac{8}{\zeta}Y'(\zeta) + 18 Y(\zeta) = - 4 Y(\zeta) \ln{(Y(\zeta})), \hspace{0.5 cm} 0< \zeta \leq 1,
}\end{equation}
with initial conditions
\begin{equation}
   \rm{Y(\zeta)|_{\zeta=0}=1, \hspace{0.4 cm} Y'(\zeta)|_{\zeta=0}=0. }
   \end{equation}
   \end{example}
The approximated solutions obtained by using proposed schemes for $\eta$ = 3 :
\begin{align*}
   \rm{ Y_{VLWC}(\zeta)} &=\rm{ e^{(-3.33067\times 10^{-16}) -  \zeta^2}},\\
 \rm{Y_{VLWT}(\zeta)} &=\rm{ e^{- \zeta^2}},\\
 \rm{ Y_{VLWG}(\zeta)} &= \rm{e^{- \zeta^2}}.
 \end{align*}
Here, $\rm{Y_{VLWT}}$ and $\rm{Y_{VLWG}}$ leads to the exact solution.  

\begin{example} 
\label{ex5}
Let us consider the nonlinear Lane-Emden type SDE \cite{wazwaz2002new}:
\begin{equation}
   \rm{ Y''(\zeta)+\frac{2}{\zeta}Y'(\zeta) - 6 Y(\zeta) =  4  Y(\zeta) \ln{(Y(\zeta))}, \hspace{0.5 cm} 0< \zeta \leq 1,
}\end{equation}
with initial conditions
\begin{equation}
   \rm{Y(\zeta)|_{\zeta=0}=1, \hspace{0.4 cm} Y'(\zeta)|_{\zeta=0}=0. }
   \end{equation}
   \end{example}
The approximated solutions obtained by using proposed schemes at $\eta$ = 3:
\begin{align*}
   \rm{ Y_{VLWC}(\zeta)} &=\rm{ e^{(3.8857\times 10^{-16}) +  \zeta^2}},\\
 \rm{Y_{VLWT}(\zeta)} &=\rm{ e^{ \zeta^2}},\\
 \rm{ Y_{VLWG}(\zeta)} &= \rm{e^{\zeta^2}}.
\end{align*}
The solutions obtained by $\rm{Y_{VLWT}}$ and $\rm{Y_{VLWG}}$ yield the exact solution.


  

\section{Results and Discussions}
\label{sec9}

\begin{figure}
	\begin{minipage}[]{0.52\textwidth}
	\centering
		\includegraphics[width=\linewidth]{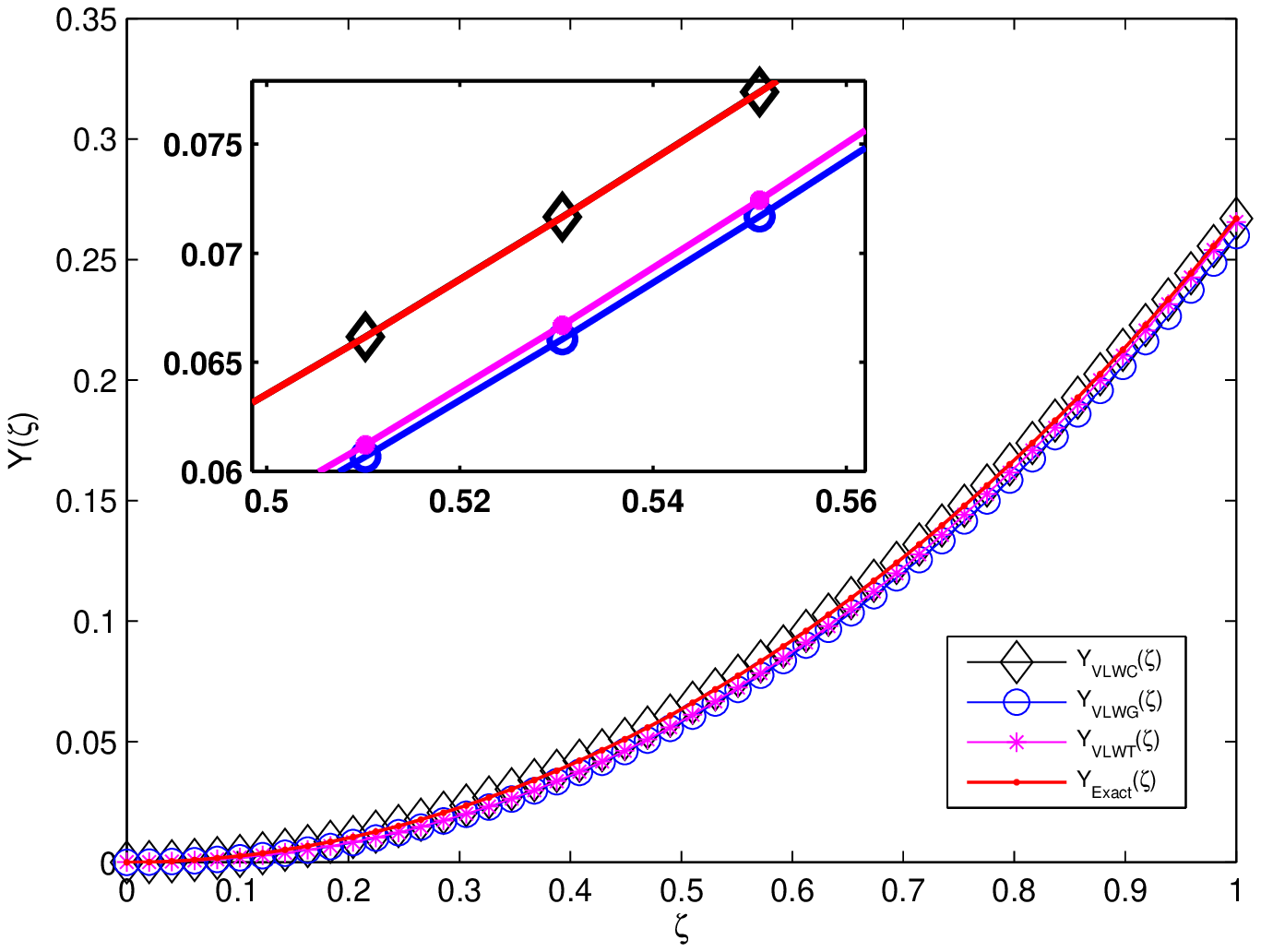}
		\subcaption{$\eta=6$}
	\end{minipage}
	\begin{minipage}[]{0.52\textwidth}
		\includegraphics[width=\linewidth]{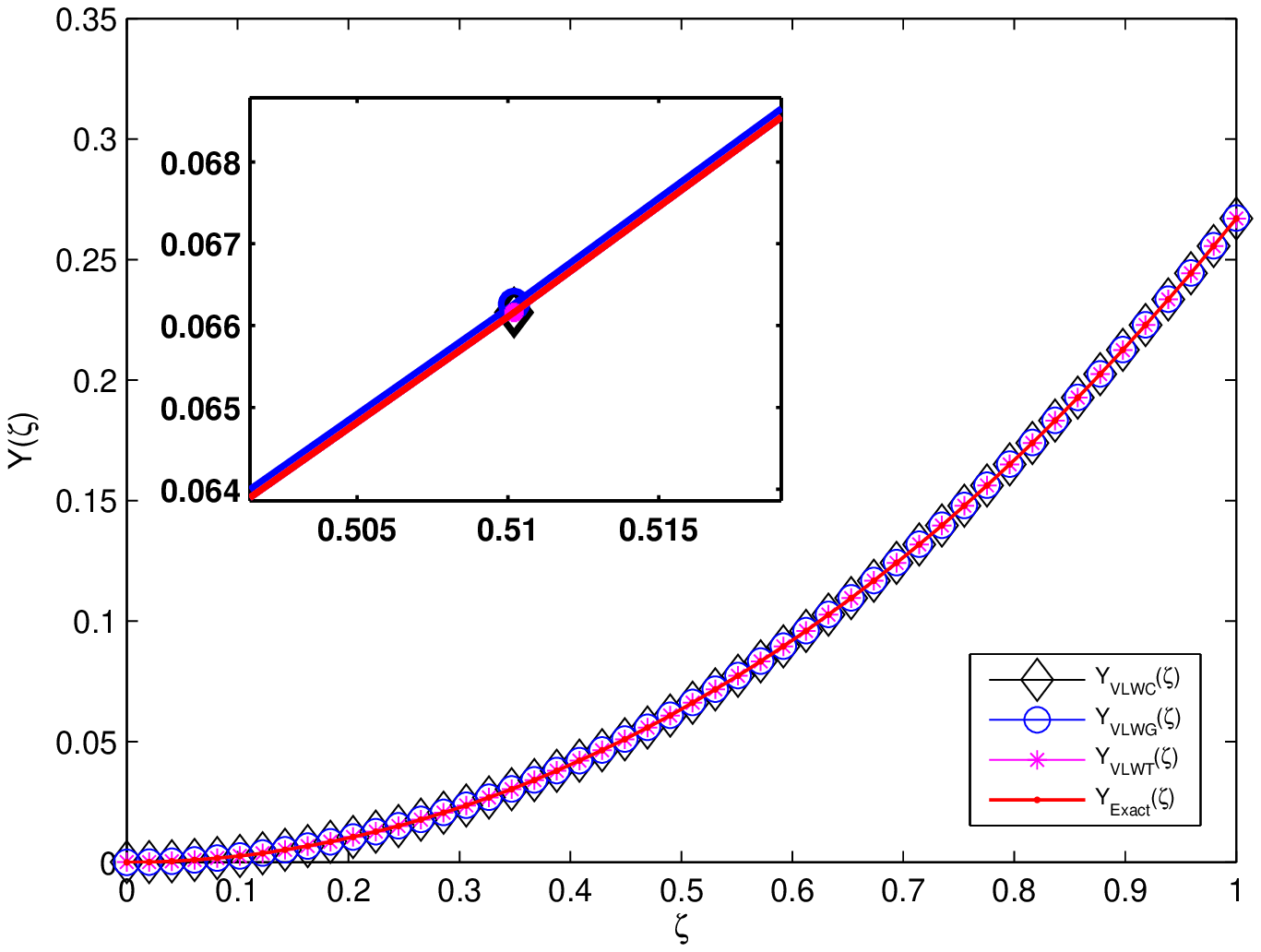}
		\subcaption{$\eta = 12$}
	\end{minipage}
 \begin{minipage}[]{0.52\textwidth}
	\centering
		\includegraphics[width=\linewidth]{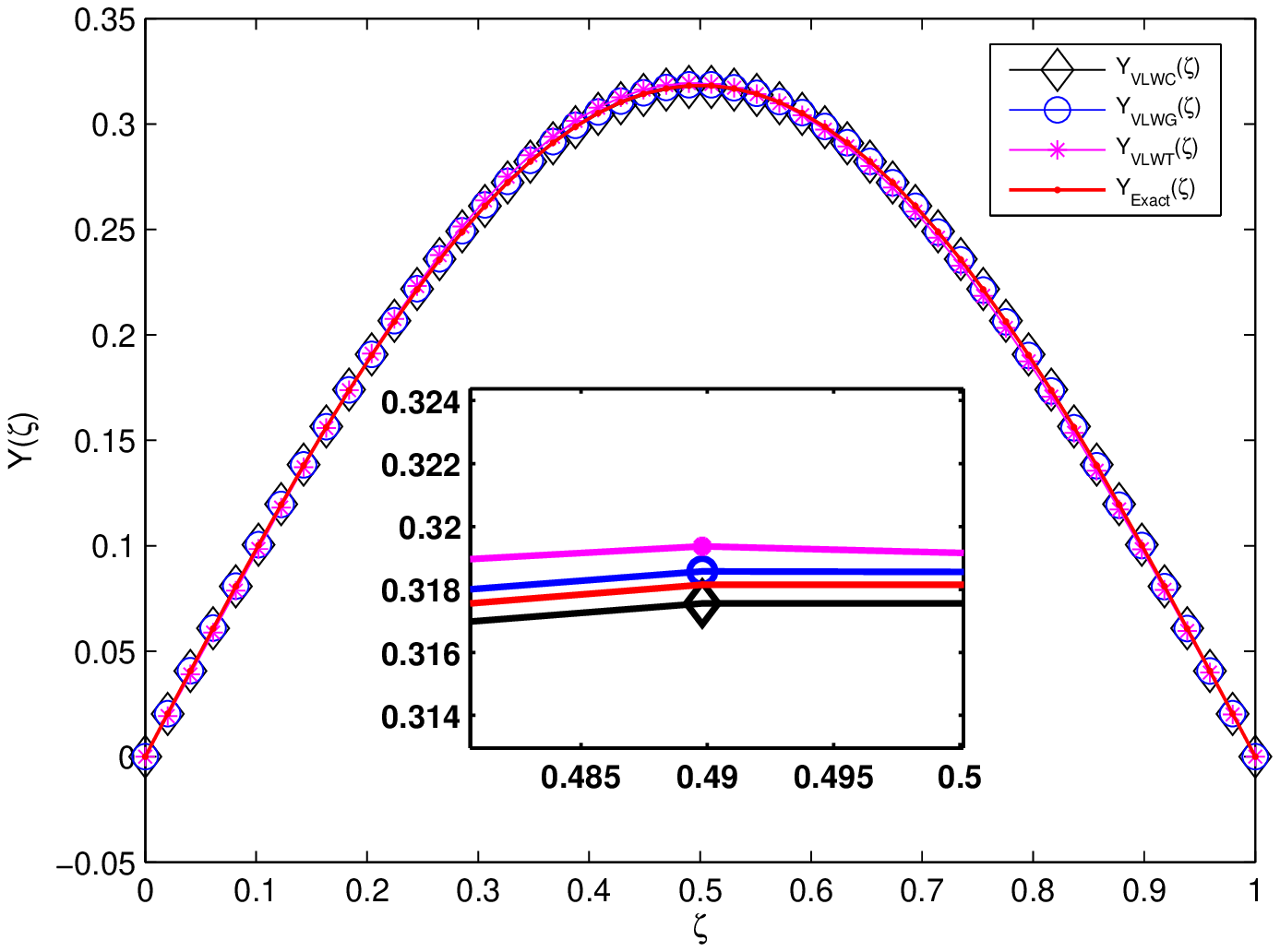}
		\subcaption{ $\eta = 6$}
	\end{minipage}
	\begin{minipage}[]{0.52\textwidth}
		\includegraphics[width=\linewidth]{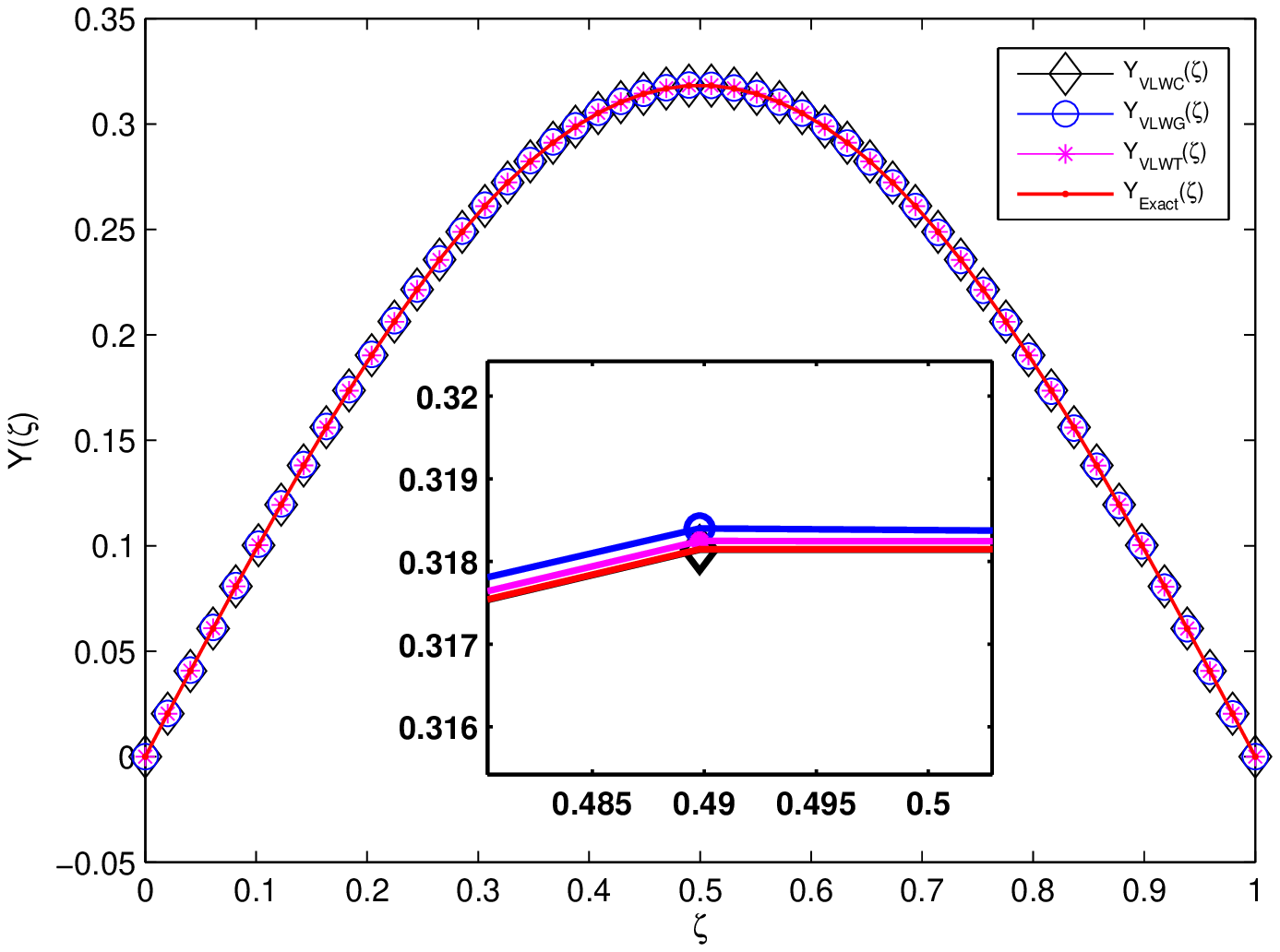}
		\subcaption{$\eta = 8$}
	\end{minipage}
		
	\caption{\label{f1} Solution curves for (a)-(b) Example \ref{ex1} and (c)-(d) Example \ref{ex2}.} 
\end{figure}

\begin{figure}
	
	\begin{minipage}[]{0.52\textwidth}
	\centering
	\includegraphics[width=\linewidth]{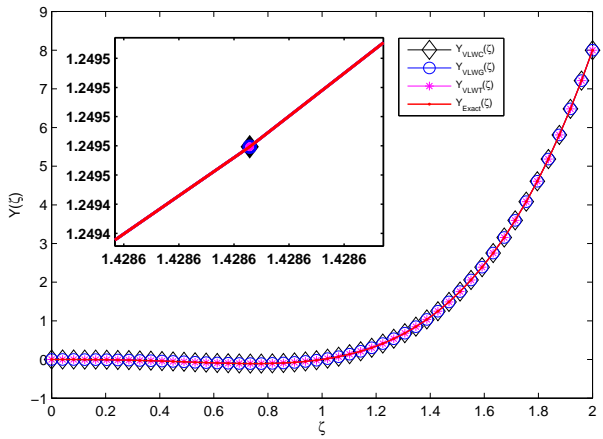}
		\subcaption{ $\eta = 5$}
	\end{minipage}
	\begin{minipage}[]{0.52\textwidth}
\centering	
 \includegraphics[width=\linewidth]{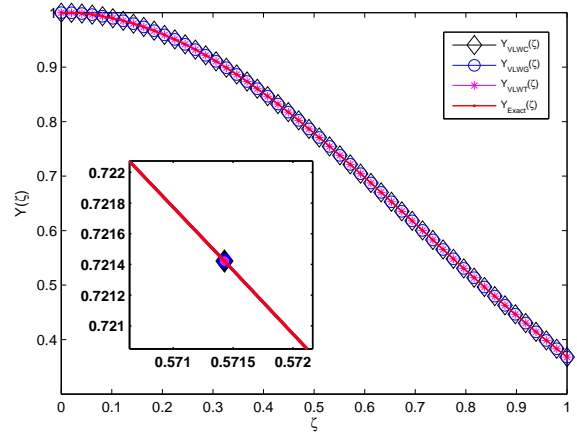}
		\subcaption{$\eta = 3$}
	\end{minipage}
 \vspace{0.5 cm}
	\begin{center}
	   \begin{minipage}{0.52\textwidth}
		\includegraphics[width=\linewidth]{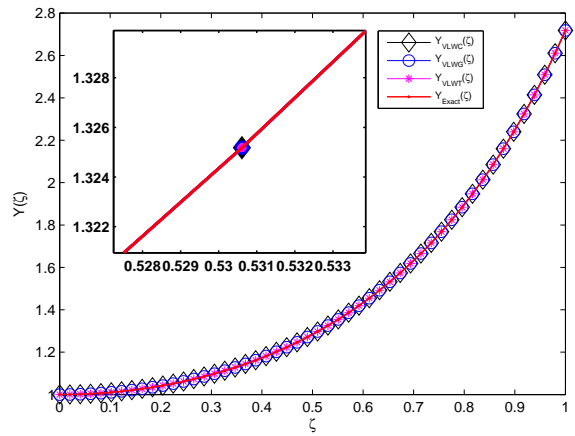}
		\subcaption{$\eta = 3$}
		\end{minipage}\hfill
	\end{center}
		
	\caption{\label{f2} Solution curves for (a) Example \ref{ex3}  (b) Example \ref{ex4} and (c) Example \ref{ex5}. } 
\end{figure}

\begin{figure}
	\begin{minipage}[]{0.52\textwidth}
	\centering
		\includegraphics[width=\linewidth]{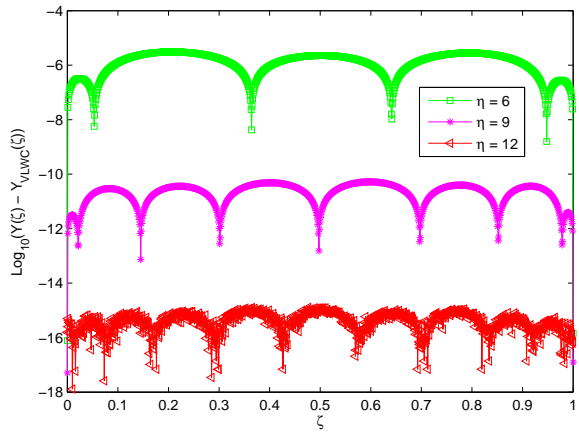}
		\subcaption{$\rm{Y_{VLWC}(\zeta)}$}
	\end{minipage}
		   \begin{minipage}{0.52\textwidth}
		\includegraphics[width=\linewidth]{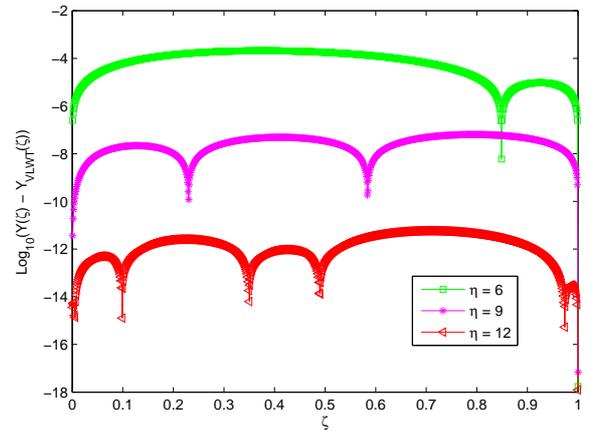}
		\subcaption{$\rm{Y_{VLWT}(\zeta)}$}
		\end{minipage}\hfill
	\vspace{0.5 cm}
\begin{center}
			\begin{minipage}[]{0.52\textwidth}
	\centering
		\includegraphics[width=\linewidth]{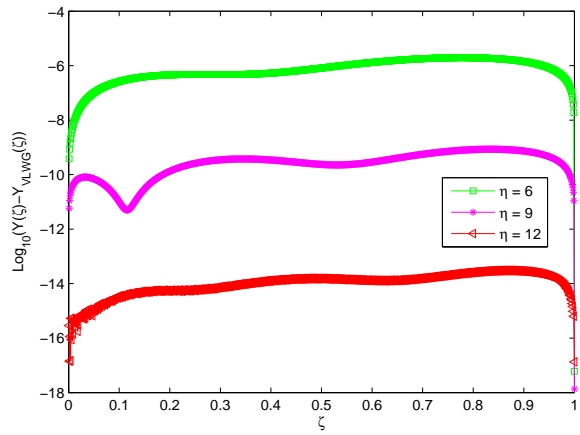}
		\subcaption{$\rm{Y_{VLWG}(\zeta)}$}
	\end{minipage}
	\end{center}
	\caption{\label{f3} Comparison of logarithmic values of absolute errors at different resolutions for Example \ref{ex1} by proposed approaches. }
\end{figure}

\begin{figure}
	\begin{minipage}[]{0.52\textwidth}
	\centering
		\includegraphics[width=\linewidth]{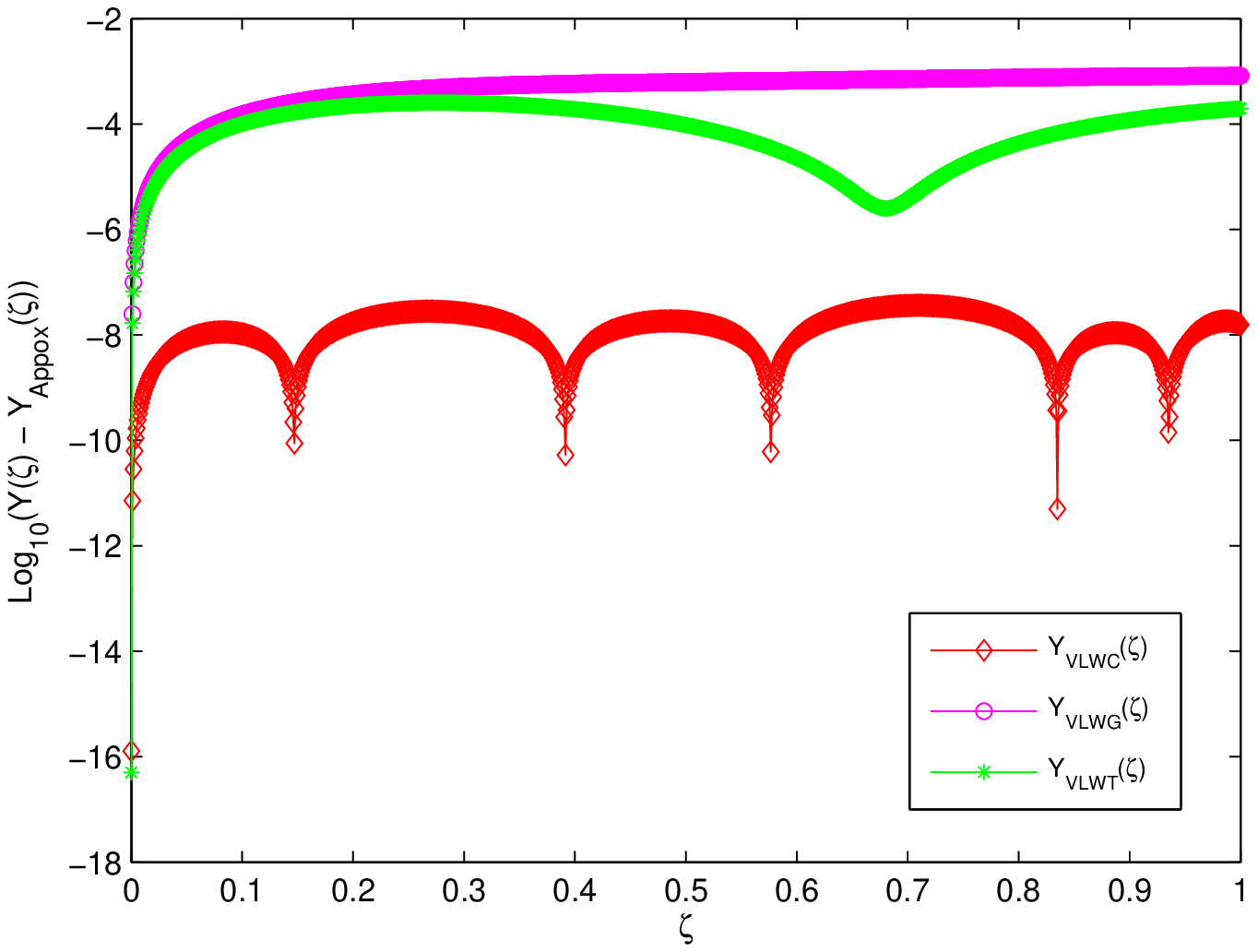}
		\subcaption{$\eta = 8$}
	\end{minipage}
	\begin{minipage}[]{0.52\textwidth}
	\centering
		\includegraphics[width=\linewidth]{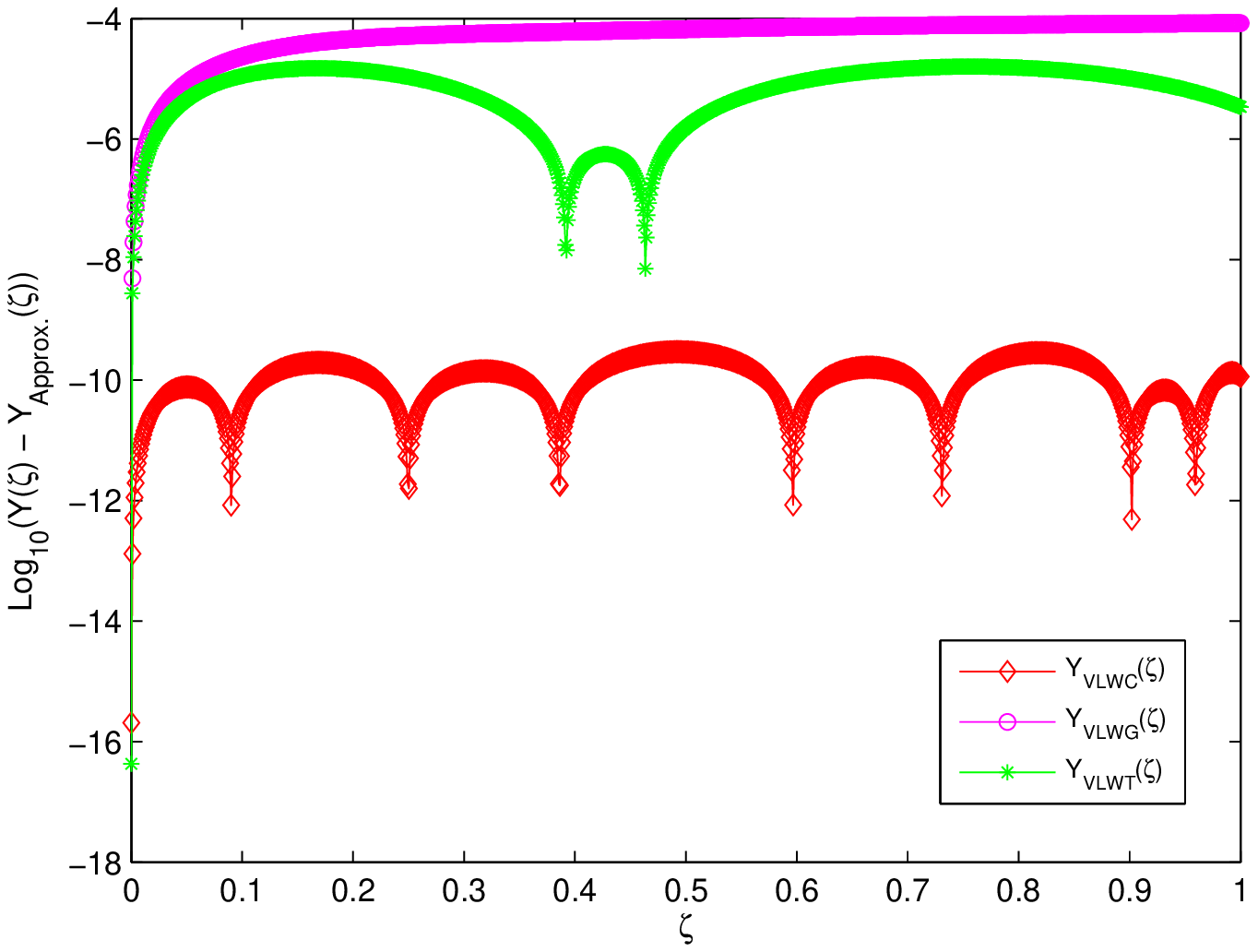}
		\subcaption{$\eta = 10$}
	\end{minipage}
	\vspace{0.5 cm}
	\begin{center}
	   \begin{minipage}{0.52\textwidth}
		\includegraphics[width=\linewidth]{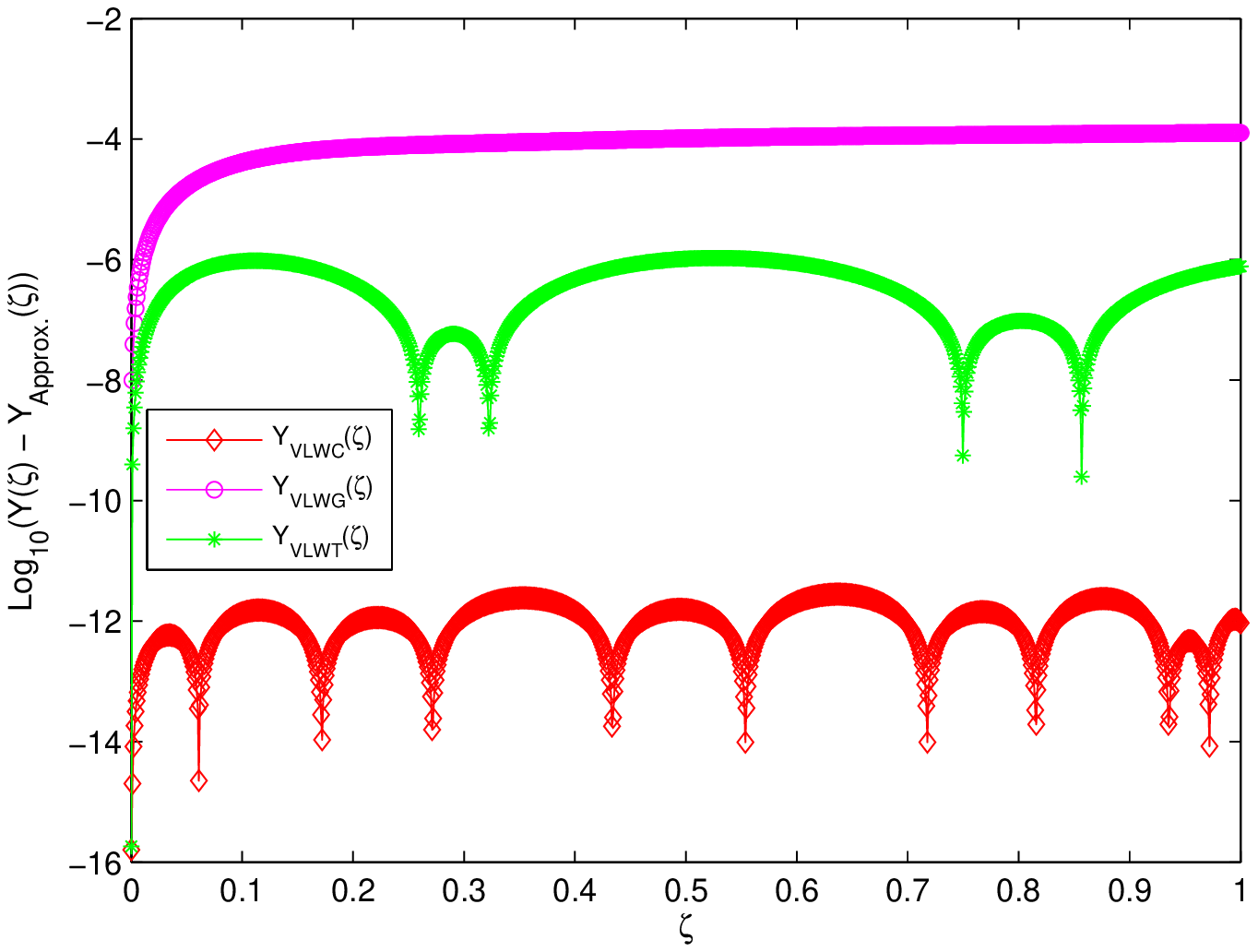}
		\subcaption{$\eta = 12$}
		\end{minipage}\hfill
	\end{center}
\caption{\label{f4} Comparison of logarithmic values of absolute errors of proposed schemes for Example \ref{ex1} at $\eta = 8, 10$ and 12. }
\end{figure}

Figure \ref{f1} demonstrates the solution curves of the exact solution and approximate solutions by the proposed numerical approaches for Example \ref{ex1} and \ref{ex2} at different resolutions. It is observed from the zoomed profile of figure \ref{f1} that as the resolution increases the approximate solutions become more accurate and overlap the exact solution curve. The solution plots for Example \ref{ex3}, \ref{ex4} and \ref{ex5} are shown in figure \ref{f2}. It can be seen from figure \ref{f2} that the proposed approaches provide accurate results at the smaller resolutions which proves that all three proposed numerical approaches are reliable. Figure \ref{f3} depicts a comparison of logarithmic values of absolute errors at various resolutions for Example \ref{ex1}, as obtained by proposed numerical approaches. It demonstrates that the errors are highest for $\eta=6$ and least for $\eta=12$, indicating that as the resolution increases, the error reduces for all three proposed numerical schemes. Figure \ref{f4} shows a comparative analysis of the presented numerical schemes at various resolutions, and it is observed that the errors are bounded for all of the proposed schemes, demonstrating the efficiency and accuracy of the schemes.\\
The absolute error obtained by the proposed methods for Example \ref{ex1} and Example \ref{ex2}  are compared in Table \ref{Table1}.
In Example \ref{ex3}, \ref{ex4} and \ref{ex5}, $\rm{Y_{VLWT}(\zeta)}$ and $\rm{Y_{VLWG}(\zeta)}$ produce the exact solution. Table \ref{Table2} compares the absolute error obtained by $\rm{Y_{VLWC}(\zeta)}$ with the existing findings. The solutions obtained from the proposed approaches are in good agreement with the existing results, demonstrating the reliability of the proposed schemes.

\begin{table}[ht]
\begin{center}
\centering
\caption{\label{Table1} Absolute error comparisons for Example \ref{ex1} and Example \ref{ex2}.} 
\begin{tabular}{p{2cm}p{3cm}p{3cm}p{3cm}p{3cm} } 
  \toprule
  \multicolumn{5}{c}{Example \ref{ex1}} \\ \hline
   $\zeta$   & Exact Solution  & $\rm{\mid Y(\zeta) - Y_{VLWC}(\zeta) \mid}$  &$\rm{\mid Y(\zeta) - Y_{VLWT}(\zeta)\mid}$ & $\rm{\mid Y(\zeta) - Y_{VLWG}(\zeta) \mid}$ \\
  \midrule
     0.1  &0.0025015      &1.3701E-12  &9.4083E-07 &4.1019E-05\\ 
    0.2 &0.0100250   &9.0015E-13  &3.9899E-07 &7.1956E-05 \\ 
    0.3 &0.0226275  &1.1882E-12  &5.3283E-08 &8.3708E-05 \\ 
    0.4 &0.0404054     &1.4496E-12  &4.7147E-07 &9.2989E-05 \\ 
    0.5  &0.0634973  &1.5533E-12  &1.0288E-06 &1.0235E-04 \\ 
    0.6  &0.0920878  &2.0826E-12  &8.6287E-07 &1.0947E-04 \\ 
     0.7  &0.1264121 &8.8321E-13 &2.4217E-07  &1.1425E-04 \\ 
    0.8 &0.1667632  &7.3194E-13  &9.5880E-08 &1.1836E-04 \\ 
    0.9 &0.2135007  &1.8588E-12 &1.9422E-07  &1.2276E-04\\ 
    1.0 &0.2670627  &9.3192E-13  &7.7329E-07 &1.2706E-04 \\
 \toprule 
  \multicolumn{5}{c}{Example \ref{ex2}} \\ 
  \midrule
   0.1 & 0.0983631  &1.7106E-04    &  1.7790E-03  & 2.0126E-04\\ 
    0.2  &0.1870978 &3.9765E-04   &   6.3775E-04 & 3.4647E-04  \\   
    0.3  &0.2575181  &1.2307E-04   & 2.6134E-03 & 5.0466E-04 \\ 
    0.4  &0.3027306 & 3.6135E-04   & 2.6338E-03  & 5.4599E-04 \\    
    0.5  &0.3183098 &5.9103E-04   & 1.0177E-03  & 4.0975E-04 \\   
    0.6  &0.3027306 &3.6135E-04   & 1.1302E-03  & 1.5710E-04 \\  
    0.7  &0.2575181 &1.2307E-04   & 2.6834E-03  & 8.3777E-05 \\   
    0.8  &0.1870978  &3.9765E-04  & 2.9582E-03 & 2.0029E-04 \\  
    0.9  &0.0983631  &1.7106E-04    & 1.8955E-03 & 1.5510E-04 \\  
  \bottomrule
\end{tabular}
\end{center}
\end{table}

\begin{table}[ht]
\begin{center}
\caption{\label{Table2}  Absolute error comparisons for Example \ref{ex3}, Example \ref{ex4} and Example \ref{ex5}. }
\begin{tabular}{p{3.5cm}p{3.5cm}p{3.5cm}p{3.5cm}} 
  \toprule 
  \multicolumn{4}{c}{Example \ref{ex3}} \\ \hline
   $\zeta$  & Exact Solution & HFC Error\cite{parand2010approximation}  & $\rm{\mid Y(\zeta) - Y_{VLWC}(\zeta) \mid}$  \\ 
  \midrule
    0.01 & -0.0000009 &5.7E-08 & 5.5603E-16   \\  
    0.10 & -0.0009000 &8.4E-08 &6.3881E-16    \\  
    0.50 & -0.0625000 &2.2E-06 &1.8318E-15    \\  
    1.00 & 0.0000000 &8.2E-07 &2.7755E-15 \\ 
     2.00 & 8.0000000 &1.7E-07 &1.7763E-15   \\  
  \toprule 
  \multicolumn{4}{c}{Example \ref{ex4}} \\ 
  \midrule
    $\zeta$  & Exact Solution & LWM Error\cite{aminikhah2013numerical}  & $\rm{\mid Y(\zeta) - Y_{VLWC}(\zeta) \mid}$  \\
  \midrule
    0.1 &0.9900498 &4.8E-06 &3.3306E-16 \\ 
    0.2  &0.9607894 & 6.8E-06 &3.3306E-16  \\  
    0.3  &0.9139311 & 8.0E-07 &3.3306E-16  \\ 
    0.4  &0.8521437 & 8.3E-06 & 3.3306E-16  \\  
    0.5  &0.7788007 & 1.2E-05 &2.2204E-16    \\ 
    0.6  &0.6976763 & 5.3E-05 &1.1102E-16    \\ 
    0.7  &0.6126263 & 2.0E-04 &1.1102E-16   \\  
    0.8  &0.5272924 & 5.9E-04 &2.2204E-16 \\  
    0.9  &0.4448580 & 1.4E-03 &5.5511E-17    \\ 
    1.0 & 0.3678794 &3.0E-03  &5.6511E-17    \\
    \toprule 
  \multicolumn{4}{c}{Example \ref{ex5}} \\ 
  \midrule
  $\zeta$  & Exact Solution & HFC Error\cite{parand2010approximation}  & $\rm{\mid Y(\zeta) - Y_{VLWC}(\zeta) \mid}$  \\
  \midrule
   0.01  &1.0001000 & 2.2E-08  &4.4408E-16  \\ 
    0.02  &1.0004000  & 1.5E-08 &2.2204E-16   \\ 
     0.05  &1.0025031 & 2.1E-08 &2.2204E-16  \\ 
    0.1 &1.0100501 & 1.7E-08  &4.4408E-16  \\ 
    0.2  &1.0408107 & 2.1E-08  &4.4408E-16  \\ 
    0.5  &1.2840254 & 3.0E-08 &4.4408E-16    \\   
    0.7  &1.6323162  & 4.2E-08  &6.6613E-16 \\  
    0.8  &1.8964808 & 5.1E-08 &2.2204E-16  \\  
    0.9  &2.2479079 & 9.2E-08 &4.4408E-16  \\  
    1.0  &2.7182818 & 8.8E-08 &8.8817E-16  \\
    \bottomrule
\end{tabular}
\end{center}
\end{table}

\section{Conclusion}
\label{sec10}
 In this work, the second-order SDEs are analyzed by using the novel numerical schemes which are based on Vieta-Lucas wavelets. The operational matrix approach combined with three weighted residual approaches is used to solve singular IVPs and singular BVPs. Simultaneously two convergence theorems are proved, and the upper bound for the tolerable error is explored. The numerical examples are taken to demonstrate the efficacy of the proposed approaches. The approximate solutions derived through the proposed methods are found to be extremely accurate, which proves the accuracy and applicability. 
 
\section{Declaration}
 \textbf{Acknowledgments}: This study is supported via funding from Prince Sattam bin Abdulaziz University project number (PSAU/2023/R/1444) .\\
 \textbf{Conflict of interest}: Authors have no conflict of interest.\\
 \textbf{Availability of data and material}: Not applicable. \\
  \textbf{Compliance with ethical standards.} \\

\end{document}